\documentclass[oneside,english]{amsart}
\usepackage[T1]{fontenc}
\usepackage[latin9]{inputenc}
\usepackage{mathrsfs}
\usepackage{amsbsy}
\usepackage{amstext}
\usepackage{amsthm}
\usepackage{amssymb}
\usepackage{wasysym}
\usepackage{esint}
\usepackage{calc}
\usepackage{hyperref}
\makeatletter

\usepackage{tikz}
\def \no{\nonumber}
\def\ve{\varepsilon}
\newcommand{\pa}{\partial}
\newcommand{\R}{\mathbb{R}}
\newcommand{\ud}{\mathrm{d}}
\newtheorem{remark}{Remark}[section]
\newtheorem{theorem}{Theorem}[section]
\newtheorem{lemma}[theorem]{Lemma}

\newtheorem*{claim}{Claim}
\newtheorem*{convention}{\textbf{Convention}}

\newtheorem{thm}{\protect\theoremname}
\usepackage{babel}

\makeatother

\providecommand{\theoremname}{Theorem}

\allowdisplaybreaks

\begin{document}
	
	\title[]{Estimates for modified $\sigma_{2}$ curvature equation on compact manifolds with boundary}

	\author{Xuezhang Chen}
	\address{Department of Mathematics \& IMS, Nanjing University, Nanjing 210093, P. R. China}
	\email{xuezhangchen@nju.edu.cn}
	\thanks{X.  Chen is partially supported by NSFC (No.12271244)}
	\author{Wei Wei}
	\address{Department of Mathematics, Nanjing University, Nanjing 210093, P. R. China}
	\email{wei\_wei@nju.edu.cn}
	\thanks{W. Wei  is partially supported by NSFC (No.12201288) , NSFJS (BK20220755), and the Alexander von Humboldt Foundation.}
	\begin{abstract}
		We  establish the global $C^2$-estimates for the modified $\sigma_2$ curvature equation with prescribed boundary mean curvature, and particularly, the local boundary $C^2$ estimates on three-manifolds.    
	\end{abstract}
	
	\maketitle
	
	\section{Introduction}
	Let $(M,g)$ be a smooth compact manifold of dimension $n \geq 3$ with boundary
	$\partial M$. A modified Schouten tensor was introduced in Gursky-Viaclovsky \cite{GV} 
	
	\[
	A_{g}^{t}=\frac{1}{(n-2)}(\mathrm{Ric}_{g}-\frac{t}{2(n-1)}R_{g}g),~~ t \in \R,
	\]
	where $\mathrm{Ric}_{g}$, $R_{g}$ are the Ricci and scalar curvature, respectively. Notice that $A_{g}^{0}$ is the Ricci curvature and $A_{g}^{1}$ is
	the Schouten tensor $A_{g}$. A modified $\sigma_{k}$ curvature is defined by 
	\[
	\sigma_{k}(g^{-1}A_{g}^{t}):=\sigma_{k}(\lambda(g^{-1}A_{g}^{t}))\quad\text{~~for~~}1\le k\le n,
	\]
	where $\lambda(g^{-1}A_{g}^{t})$ is the eigenvalue of $g^{-1}A_{g}^{t}$
	and $\sigma_{k}(\lambda)$ is the elementary symmetric function defined as
	\[
	\sigma_{k}(\lambda)=\sum_{1\le i_{1}<\cdots<i_{k}\le n}\lambda_{i_{1}}\cdots\lambda_{i_{k}}.
	\]

	Let $g_{u}=e^{-2u}g$, then
	\[
	A_{g_{u}}^{t}=A_{g}^{t}+\nabla^{2}u+\frac{1-t}{n-2}\Delta ug+\nabla u\otimes\nabla u-\frac{2-t}{2}|\nabla u|^{2}g
	\]
	and
	\[
	\sigma_{2}(g^{-1}A_{g_{u}}^{t})=\frac{1}{2}(|\mathrm{tr}_{g}A_{g_{u}}^{t}|^{2}-|A_{g_{u}}^{t}|_{g}^{2}).
	\]
	
		The well-known cone condition $\Gamma_k^+:=\{\lambda=(\lambda_1,\cdots, \lambda_n);  \sigma_1(\lambda)>0,\cdots, \sigma_k(\lambda)>0\}$ is introduced to ensure the ellipticity of the equation. 

	Let $\tilde g=g \big|_{T \pa M}$ be the induced metric and $\vec n$ be the inward unit normal on $\pa M$. We define the second fundamental form by $L=-\nabla \vec n \big|_{T\pa M}$ and its trace-free part by $\mathring{L}=L-h_g \tilde g$ on $\pa M$, where $h_g=-\mathrm{tr}_{\tilde g}(\nabla \vec n)/(n-1)$ is the mean curvature.

	For $t\le1$, we consider 
	\begin{equation}
		\begin{cases}
			\sigma_{2}^{\frac{1}{2}}(g^{-1}A_{g_{u}}^{t})=f(x,u), \quad \lambda(g^{-1}A_{g_{u}}^{t})\in \Gamma_2^+ &\text{in}\quad M,\\
			\frac{\partial u}{\partial \vec{n}}+h_{g}=c(x,u) &\text{on}\quad \partial M,
		\end{cases}\label{eq:general equation}
	\end{equation}
where $f(x,u)$ is a positive function, $c(x,u)$ is a non-negative function.

	Local $C^2$ estimates are the essential ingredient for the existence of fully nonlinear equation and  have been studied extensively.  Local $C^2$ estimates for equation \eqref{eq:general equation} on closed manifolds have been investigated by various authors for $t=1$. We refer to \cite{Guan-Wang0, LiYY, Wang, Chen2} etc. Compared to the perfect $C^2$ estimates for $\sigma_k$ curvature equation on closed manifolds, the boundary $C^2$ estimates are largely open, especially for the non-umbilic boundary.  Under assumption that $\partial M$ is umbilic and $(M, g)$ is locally conformally flat near $\partial M$,  Jin-Li-Li \cite{Jin-Li-Li} obtained local $C^2$ estimates and then the existence of $\sigma_k$ curvature equation. Using a different technique, S. Chen \cite{Chen3} obtained boundary local $C^2$ estimates for the $\sigma_2$ curvature equation on four-manifolds with umbilic boundary when $c=0$.  Jin \cite{Jin} obtained boundary local $C^2$ estimates for manifolds with totally geodesic boundary.  In \cite{CW}, the authors have obtained the local $C^2$ estimates for the $\sigma_2$ curvature equation on compact manifolds with general boundary for non-negative function $c$, the method of which is distinct from the proof in this paper. For $t\le 1$,  He-Sheng \cite{HS2013} established the $C^2$ estimates for modified $\sigma_k$ curvature equation on manifolds with totally geodesic boundary for $c=0$. See \cite{Duncan-Luc,SY1,SY2} for more information related with modified equation.

	We now investigate the global $C^{2}$ estimates for
	modified $\sigma_{2}$ curvature equation on manifolds with general boundary,  
	and the local $C^{2}$ estimates for $\sigma_{2}$ curvature
	on three-manifolds with boundary, contributing to  an existence result for the
	$\sigma_{2}$-Yamabe equation with vanishing mean curvature under certain assumption.

	\begin{thm}	\label{thm:global C2 estimate}
		 Suppose $f(x,z): M \times \R \to \R_+$ is a $C^2$ function and $c(x,z): \partial M \times \R \to \R$ is a nonnegative $C^3$ function. Let $u$ be a $C^{4}$ solution to equation (\ref{eq:general equation}) for $ t\le 1$.
		Then there holds 
		\[
		\sup_{M}|\nabla^{2}u|\le C,
		\]
		where $C$ is a positive constant depending on $g,c$,
		$|u|_{C^1(M)}$, $\sup_{M}(f(x,u)+|f_{x}|+|f_{z}|+|f_{xx}|+|f_{xz}|+|f_{zz}|)$
		and $\sup_{M}\frac{1}{f}.$ 
	\end{thm}
	
	We emphasize that for $t<1$, the $C^2$ estimates hold regardless of  the sign of $c(x,u)$ and the dependence on $\sup_{M}\frac{1}{f}.$

	The proof of Theorem \ref{thm:global C2 estimate} follows the classical method developed by Lions-Trudinger-Urbas \cite{LTU} for the Monge-Ampere equation on a strictly convex Euclidean domain. This method consists of two key steps: First, reducing global estimates to double normal derivative on the boundary; second, deriving the double normal derivative on the boundary. Similar techniques had been applied by Ma-Qiu \cite{MQ} to the $k$-Hessian equation on strictly convex domains. See \cite{JT} for the global $C^2$ estimates. Most existing results heavily depend on the convexity of the boundary.

	The boundary $\sigma_k$ curvature equation is challenging compared with other fully nonlinear equations. The term $\sum_{i,j} F^i_j d_g(x,\partial M)^j_i$ can not contribute a good term to control other terms from  taking the first derivatives to the equation, making the computation of double normal derivative on $\pa M$ more complicated.
	To address the challenge, new auxiliary functions and certain geometric insights come into play, enabling the resolution of the intricate term above.  	$\vspace{4pt}$

	With Theorem \ref{thm:global C2 estimate}, we are able to prove the following boundary 
$\sigma_2$ eigenvalue problem. 
	\begin{thm}
		\label{theorem:sigma_2_1st_bdry_eigenvalue-1}  Suppose $\lambda(g^{-1}A^t_{g})\in\Gamma_{2}^{+}$ for $t\le 1$
		and $h_{g}\ge0$ on $\pa M$. Then for any $t\le 1$, there exist a unique constant
		$\lambda>0$ and $u\in C^{\infty}(\overline{M})$ such that $g_{u}=e^{-2u}g$
		satisfies
		\[
		\begin{cases}
			\sigma_{2}(g^{-1}A^t_{g_{u}})=\lambda & \quad\mathrm{~~in~~}M,\\
			h_{g_{u}}=0 & \quad\mathrm{~~on~~}\partial M.
		\end{cases}
		\]
	\end{thm}
	
	The $\sigma_k$ eigenvalue problem emerged in \cite{Guan-Wang, Ge-Lin-Wang,Ge-Wang1} etc. In particular, Ge-Lin-Wang \cite{Ge-Lin-Wang} demonstrated that the sign of $\sigma_2$-Yamabe constant coincides with that of $\sigma_2$ eigenvalue. Wang \cite{G.F.Wang} showed that the eigenvalue of Dirac operator can be estimated by  the $\sigma_2$  curvature on spin four-manifolds utilizing the $\sigma_2$ eigenvalue.  Duncan and Nguyen \cite{Duncan-Luc} studied the modified $\sigma_k$  eigenvalue problem on closed manifolds and thereby demonstrated some geometric consequences. For nonlinear eigenvalue problem with boundary, various works of  fully nonlinear equations in this direction have been developed in recent years, we refer to \cite{MQ} and related papers.

%
%

	We next restrict consideration to three-manifolds, where applications related to $\sigma_2$ curvature have been explored, offering a dedicated geometric description parallel to four-manifolds \cite{CGY1, CGY2}. Gursky-Viaclovsky \cite{GV01} classified three-manifolds associated to $\sigma_2$ curvature. Based on the nonlinear eigenvalue and Hamilton's results, Ge-Lin-Wang \cite{Ge-Lin-Wang} established a three dimensional sphere theorem. Employing a different method, Catino-Djadli \cite{CD} independently proved the same result. Ge-Wang \cite{Ge-Wang2} introduced a new conformal invariant and obtained a new characterization of the round three-sphere.
	
	Concerning the $C^2$ estimates for $\sigma_2$ curvature equation on  three-manifolds with boundary, the estimate of local double normal derivative  exhibits a remarkable feature: It can be directly obtained without the estimates of  second tangential derivatives on the boundary. Through a reduction argument, local $C^2$ estimates near the boundary can be derived, and are crucial for the existence of boundary $\sigma_2$ Yamabe problem.

		\begin{thm}\label{thm:three-double normal}
		On $(M^{3},g)$, assume as Theorem \ref{thm:global C2 estimate}. Let $u$ be
		a $C^{3}$ solution to equation (\ref{eq:general equation}). Then for
		any $\mathcal{O}\subset\mathcal{O}_{1}\subset\overline{M}$ there
		holds $$\nabla^2 u(\nabla d_g(x,\partial M), \nabla d_g(x,\partial M))\le C\qquad \mathrm{~~on~~}\quad  \mathcal{O}\cap\partial M,$$ where $C$ is
		a positive constant depending on $g,f$, $\sup_{\mathcal{O}_{1}\cap\partial M}(c(x,u)+|c_{x}|+|c_{z}|+|c_{xx}|+|c_{xz}|+|c_{xz}|)$,
		$\sup_{\mathcal{O}_{1}}|\nabla u|$, $\sup_{\mathcal{O}_{1}}(f(x,u)+|f_{x}|+|f_{z}|+|f_{xx}|+|f_{xz}|+|f_{xz}|)$
		and $\sup_{\mathcal{O}_{1}}\frac{1}{f}.$
	\end{thm}

	Motivated by Gursky-Viaclovsky \cite{GV0} and Guan-Wang \cite{Guan-Wang3}, we introduce a conformal invariant  by
	\[
	\Lambda_{2}(M^{3},\partial M, [g])=\sup\left\{ \mathrm{vol}(e^{-2u}g); g_{u}=e^{-2u}g\quad  \text{satisfies condition (\ref{curvature condition})} \right\},
	\]
	where
	\begin{align}\tag{CC}\label{curvature condition}
		\begin{cases}
			\sigma_{2}^{\frac{1}{2}}(g_{u}^{-1}A_{g_{u}})\ge\sigma_{2}^{\frac{1}{2}}(g_{\mathbb{S}^{3}}^{-1}A_{\mathbb{S}^{3}}),\lambda(g^{-1}A_{g_{u}})\in\Gamma_{2}^{+} & \text{~~in~~}\:M^3,\\[3pt]
			h_{g_{u}}\ge0 & \text{~~on~~}\,\,\partial M.
		\end{cases}
	\end{align}
	
	Under certain assumption of this conformal invariant, we obtain the following 
	existence result.
	\begin{thm}
		Suppose $\lambda(g^{-1}A_{g})\in\Gamma_{2}^{+}$, the mean curvature $h_{g}\ge0$ on $\partial M$, and 	 $\Lambda_{2}(M^{3},\partial M, [g])<\frac{1}{2}\mathrm{vol}(\mathbb{S}^{3},g_{\mathbb{S}^{3}}).$
		Then  there exists a smooth solution to
			\begin{align*}
		\begin{cases}
			\sigma_{2}^{\frac{1}{2}}(g^{-1}A_{g_{u}})=\sigma_{2}^{\frac{1}{2}}(g_{\mathbb{S}^{3}}^{-1}A_{\mathbb{S}^{3}})e^{-2u} &\mathrm{in}\quad M^{3},\\
			h_{g_{u}}=0 &\mathrm{on}\quad \partial M.
		\end{cases}
	\end{align*}
	\end{thm}
	
	Very recently,  the authors  \cite{CW} had proved the existence of $\sigma_2$ curvature equation on compact three- or four-manifolds with totally non-umbilic boundary. 	On three-manifolds with boundary, Gursky-Viaclovsky \cite{GV0}  introduced the conformal invariant 
	\[
	\Lambda_{2}(M^{3}, [g])=\sup\left\{ \mathrm{vol}(e^{-2u}g); \sigma_{2}^{\frac{1}{2}}(g_{u}^{-1}A_{g_{u}})\ge\sigma_{2}^{\frac{1}{2}}(g_{\mathbb{S}^{3}}^{-1}A_{\mathbb{S}^{3}}),\,\,\lambda(g^{-1}A_{g_{u}})\in\Gamma_{2}^{+}\right\}
	.\]
	Using
	 $$\sigma_{2}^{\frac{1}{2}}(g_{u}^{-1}A_{g_{u}})\ge\sigma_{2}^{\frac{1}{2}}(g_{\mathbb{S}^{3}}^{-1}A_{\mathbb{S}^{3}}),\,\,\lambda(g^{-1}A_{g_{u}})\in\Gamma_{2}^{+},$$
	we know that $\mathrm{Ric}_{g}\ge g$ and $R_g\ge R_{\mathbb{S}^3}$. Based on Bray's football theorem \cite{Bray}, Gursky-Viaclovsky proved that $\mathrm{vol}(M^3, g)\le \mathrm{vol}(\mathbb{S}^3,g_{\mathbb{S}^3})$ and thus, $\Lambda_{2}(M^{3}, [g])\le \mathrm{vol}(\mathbb{S}^3, g_{\mathbb{S}^3}).$

		
		\begin{convention} 
		The Latin letters like $1\leq i,j,k \leq n$ stand for the full indices; the Greek letters like $1 \leq \alpha,\beta,\gamma\leq n-1$ stand for tangential indices. 
		
	\end{convention}
		
	We organize the paper as follows. In Section \ref{Sect2}, we prove the global $C^2$ estimates, dividing the proof into two subsections. In Subsection \ref{Subsect2.1}, we demonstrate that  local second derivatives can be dominated by double normal derivative on the boundary. In Subsection \ref{Subsect2.2}, we construct an auxiliary function and obtain the double normal derivative using the maximum principle. With the global estimates  in Section \ref{Sect2}, we establish Theorem \ref{theorem:sigma_2_1st_bdry_eigenvalue-1} in Subsection \ref{Subsect2.3}. In Section \ref{Sect3}, we obtain the local double normal derivative on the boundary of a three-manifold, which results in the local $C^2$ estimates  deducing from the conclusion in Subsection \ref{Subsect2.1}. Imposing some natural geometric condition, we establish the existence of $\sigma_2$ Yamabe problem on three-manifolds with boundary.

	\section{The global $C^2$-estimate on manifolds}\label{Sect2}
	In this section, we will  reduce the $C^2$ estimates to the double normal derivative on the boundary for a general  function $F(g^{-1}W)$,  where $W:=A_{g_{u}}^{t}$ for short  and $F$ is a concave  function satisfying  the following conditions: Under the  geodesic normal coordinates,
	\begin{equation}\label{symmetric}
		\frac{\partial F }{\partial W_{ij}}\quad \text{is a positive definite matrix}
	\end{equation}
	and 
	\begin{equation}\label{lower bound of sum}
		\sum_{i=1}^n\frac{\partial F}{\partial W_{ii}}\ge C_0, \quad \text{for some positive constant}\,\, C_0
	\end{equation}
	and
	\begin{equation}\label{concave condition}
		\sum_{i,j,s,t=1}^n\frac{\partial^2 F}{\partial W_{ij}\partial W_{st}}\eta_{ij}\eta_{st}\le 0, \quad \text{for any symmetric matrix}\,\, (\eta_{ij})_{n\times n}.
	\end{equation}	
	In the following, we denote  $\tilde{F}^{ij}=\frac{\partial F }{\partial W_{ij}}$,  $\tilde{F}^{ij,rs}:=\frac{\partial^2 F}{\partial W_{ij}\partial W_{rs}}$ and 
	\[
	\tilde{P}^{ij}=\tilde{F}^{ij}+\frac{1-t}{n-2}\big(\sum_{k=1}^{n}\tilde{F}^{kl}g_{kl}\big)g^{ij}.
	\]

	Now we start with the proof of global $C^2$ estimates.

	\subsection{The reduction argument}	\label{Subsect2.1}
	For brevity, we let $\frac{\partial u}{\partial \vec n}:=u_{n}=c(x,u)-h_{g}:=\varphi(x,u)$ on $\partial M$. 
	\begin{thm}
		\label{Thm:reduction} 
		Assume that for $1+\frac{n(1-t)}{n-2}>0$, $F$ satisfies \eqref{symmetric}\eqref{lower bound of sum}\eqref{concave condition}.
		Suppose $u$ is a $C^4$ solution to 
		\begin{align}
			\begin{cases}
				F(g^{-1}A_{g_{u}}^{t})=f(x,u) &\quad \mathrm{~~in~~}\quad M,\\
				\frac{\partial u}{\partial \vec{n}}+h_{g}=c(x,u) &\quad\mathrm{~~on~~}\quad \partial M,
			\end{cases}\label{eq:general equation F}
		\end{align}
		where $f(x, u)$ is a non-negative $C^2$ function, $c(x,u)$ is a  $C^3$ function
		and $h_{g}$ is the mean curvature. Let $\mathcal{O},\mathcal{O}_{1}\subset\overline{M}$
		be two domains near $\pa M$ such that $\mathcal{O}\Subset\mathcal{O}_{1}$
		and $\mathcal{O}\cap\pa M\neq\emptyset$, we have 
		\begin{align*}
			\max_{x\in\overline{\mathcal{O}}}\max_{\substack{\zeta\in T_{x}\overline{M}\\
					|\zeta|=1
				}
			}\nabla^{2}u(\zeta,\zeta) & \le C+\max_{\overline{\mathcal{O}_{1}}\cap\partial M}\nabla^{2}u(\vec n,\vec n),
		\end{align*}
		where $C$ is a positive constant depending on $g$, $|u|_{C^1(\mathcal{O}_{1})}$, $c$
		and
		$
		\sup_{\mathcal{O}_{1}}(f(x,u)+|f_{x}|+|f_{z}|+|f_{xx}|+|f_{xz}|+|f_{xz}|).
		$
		
	\end{thm}
	
	\begin{proof}
		Let $\eta$ be a smooth cut-off function in $\mathcal{O}_{1}$ such
		that $\eta=1$ on $\mathcal{O}$ and denote by $\mathbb{S}\overline{M}$
		the spherical tangent bundle on $\overline{M}$. Given $(x,\zeta)\in\mathbb{S}\overline{M}$,
		we define 
		\begin{align*}
			G(x,\zeta):= & \eta e^{ad}\bigg[\nabla^{2}u(\zeta,\zeta)-2\langle\zeta,\nabla d\rangle\left(\langle\zeta^{\top},\nabla\varphi\rangle-\nabla^{2}d(\nabla u,\zeta^{\top})\right)\\
			& \qquad~~+B\langle\nabla u,\nabla d\rangle+|\nabla u|^{2}\bigg],
		\end{align*}
		where $a\in\R_{+}$ is to be determined later, $d(x):=d_{g}(x,\partial M)$,
		$\zeta^{\top}=\zeta-\langle\zeta,\nabla d\rangle\nabla d$ and $B=2(\sup_{\mathcal{O}_{1}\cap\partial M}|c(x,u)|+|h_g|+n\sup_{\mathcal{O}_{1}\cap\partial M}|L|)$.

		By definition of $\zeta^{\top}$ we have 
		\begin{align*}
			\langle\zeta^{\top},\nabla\varphi\rangle-\nabla^{2}d(\nabla u,\zeta^{\top})=\langle\zeta,\nabla\varphi\rangle-\langle\zeta,\nabla d\rangle\langle\nabla d,\nabla\varphi\rangle-\nabla^{2}d(\nabla u,\zeta).
		\end{align*}
		
		Suppose 
		\[
		\max_{x\in\overline{\mathcal{O}}}\max_{\substack{\zeta\in T_{x}\overline{M}\\
				|\zeta|=1
			}
		}G(x,\zeta)=G(x_{1},\zeta_{1})\qquad\mathrm{~~for~~some~~}(x_{1},\zeta_{1})\in\mathbb{S}\overline{M}.
		\]
		
		Without loss of generality, we assume that $G(x_{1},\zeta_{1})\gg1$,
		otherwise it is done.
		
		\vskip 4pt \emph{Case 1:} $x_{1}\in\mathcal{O}_{1}\cap\partial M$.
		\vskip 4pt
		
		\emph{Case 1.1:} $\zeta_{1}=\alpha\nabla d+\beta\tau$ for some $\tau\in T_{x_{1}}(\partial M)$
		with $|\tau|=1$, $\alpha=\langle\zeta_{1},\nabla d\rangle$ and $\beta=\langle\zeta_{1},\tau\rangle$
		satisfy $\alpha^{2}+\beta^{2}=1$.
		
		\vskip 4pt
		
		Notice that $\vec n=\nabla d$ on $\pa M$ and 
		\begin{align*}
			\nabla^{2}u(\tau,\vec n) & =\tau(u_{n})-(\nabla_{\tau}\vec n)u\\
			& =\langle\tau,\nabla\varphi\rangle-\nabla^{2}d(\nabla u,\vec n)\qquad\mathrm{~~on~~}\pa M
		\end{align*}
		by virtue of the boundary condition that $u_{ n}=\varphi$. Thus,
		we obtain 
		\begin{align*}
			G(x_{1},\zeta_{1})= & \eta e^{ad}\bigg[\alpha^{2}\nabla^{2}u(\vec n,\vec n)+\beta^{2}\nabla^{2}u(\tau,\tau)+2\alpha\beta\nabla^{2}u(\tau,\vec n)\\
			& \qquad~~-2\alpha\beta\left(\langle\tau,\nabla\varphi\rangle-\nabla^{2}d(\nabla u,\tau)\right)+B\langle\nabla u,\nabla d\rangle+|\nabla u|^{2}\bigg]\Bigg|_{x_{1}}\\
			= & \eta e^{ad}\bigg[\alpha^{2}\nabla^{2}u(\vec n,\vec n)+\beta^{2}\nabla^{2}u(\tau,\tau)+B\langle\nabla u,\nabla d\rangle+|\nabla u|^{2}\bigg]\Bigg|_{x_{1}}\\
			= & \alpha^{2}G(x_{1},\vec n)+\beta^{2}G(x_{1},\tau)\le\alpha^{2}G(x_{1},\vec n)+\beta^{2}G(x_{1},\zeta_{1}),
		\end{align*}
		which yields that 
		\[
		G(x_{1},\zeta_{1})\le G(x_{1},\vec n).
		\]
		
		This forces $\zeta_{1}$ to be either $\vec n$ or a unit vector in $T_{x_{1}}(\partial M)$.
		
		\vskip 4pt \emph{Case 1.2:} $\zeta_{1}\in T_{x_{1}}(\partial M)$.
		\vskip 4pt
		
		Under Fermi coordinates around $x_{1}$, the metric can be expressed
		as 
		\[
		g=\ud x_{n}^{2}+g_{\alpha\beta}\ud x_{\alpha}\ud x_{\beta},
		\]
		where $(x_{1},\cdots,x_{n-1})$ is the geodesic normal coordinates
		on $\partial M$ and $\partial_{x_{n}}=\nabla d$ is the inward unit
		normal on $\partial M$. Without loss of generality, we assume $\zeta_{1}=\frac{\pa}{\pa{x_{1}}}$
		at $x_{1}$. We choose the tangent vector field $\frac{1}{\sqrt{g_{11}}}\partial_{x_{1}}$
		as an extension of $\zeta_{1}$ to the interior and still denote by
		$\zeta_{1}$. On $\partial M$, the Christoffel symbols are $\Gamma_{\alpha\beta}^{n}=L_{\alpha\beta},$
		$\Gamma_{\alpha n}^{\beta}=-L_{\alpha\gamma}g^{\gamma\beta},$ $\Gamma_{in}^{n}=0$
		and $\Gamma_{nn}^{i}=0$.
		
		Observe that 
		\begin{align*}
			u_{\alpha n}=\partial_{\alpha}u_{n}+L_{\alpha}^{\gamma}u_{\gamma}=\partial_{\alpha}(c(x,u)-h_g)+L_{\alpha}^{\gamma}u_{\gamma},
		\end{align*}
		\begin{align*}
			u_{\alpha\beta n}= & L_{\beta}^{\gamma}u_{\gamma\alpha}+L_{\alpha}^{\gamma}u_{\gamma\beta}-L_{\alpha\beta}u_{nn}+\partial_{\beta}\partial_{\alpha}(c(x,u))\\
			& +L_{\alpha;\beta}^{\gamma}u_{\gamma}+L_{\alpha}^{\gamma}L_{\beta\gamma}u_{n}-h_{g;\alpha\beta}+R_{\alpha\beta n}^{i}u_{i}\\
			= & L_{\beta}^{\gamma}u_{\gamma\alpha}+L_{\alpha}^{\gamma}u_{\gamma\beta}-L_{\alpha\beta}u_{nn}+c_{z}u_{\alpha\beta}\\
			& +L_{\alpha;\beta}^{\gamma}u_{\gamma}+L_{\alpha}^{\gamma}L_{\beta\gamma}u_{n}-h_{g;\alpha\beta}+R_{\alpha\beta n}^{i}u_{i}\\
			& +c_{x_{\alpha}x_{\beta}}+c_{x_{\alpha}z}u_{\beta}+c_{zx_{\beta}}u_{\alpha}+c_{zz}u_{\beta}u_{\alpha}+c_{z}L_{\alpha\beta}u_{n}
		\end{align*}
		where
		\begin{align*}
			\partial_{\beta}\partial_{\alpha}(c(x,u)) & =\partial_{\beta}(c_{x_{\alpha}}+c_{z}u_{\alpha})\\
			& =c_{x_{\alpha}x_{\beta}}+c_{x_{\alpha}z}u_{\beta}+c_{zx_{\beta}}u_{\alpha}+c_{zz}u_{\beta}u_{\alpha}+c_{z}(u_{\alpha\beta}+L_{\alpha\beta}u_{n}).
		\end{align*}
		In particular, we have 
		\begin{align*}
			\partial_{n}(u_{11})= & 2\Gamma_{1n}^{\alpha}u_{\alpha1}+u_{11n}\\
			= & u_{11n}-2L_{1}^{\alpha}u_{\alpha1}\\
			= & -L_{11}u_{nn}+c_{z}u_{11}+L_{1;1}^{\gamma}u_{\gamma}+L_{1}^{\gamma}L_{1\gamma}u_{n}\\
			& +c_{x_{1}x_{1}}+c_{x_{1}z}u_{1}+c_{zx_{1}}u_{1}+c_{zz}u_{1}u_{1}+c_{z}L_{11}u_{n}-h_{g;11}+R_{11n}^{i}u_{i}.
		\end{align*}
		
		Notice that 
		\[
		G(x,\zeta_{1})=\eta e^{ad}\bigg(\frac{u_{11}}{g_{11}}+Bu_{n}+|\nabla u|^{2}\bigg)\qquad\mathrm{~~near~~}x_{1}.
		\]
		Then we have 
		\begin{align}
			0\ge & G_{n}(x_{1},\zeta_{1})\label{eq:hopf}\\
			= & \eta\bigg[a(u_{11}+Bu_{n}+|\nabla u|^{2})+\pa_{n}(u_{11})-u_{11}\pa_{n}(g_{11})+Bu_{nn}+2u^{\alpha}u_{\alpha n}+2u_{n}u_{nn}\bigg]\no\\
			= & \eta\bigg[a(u_{11}+Bu_{n}+|\nabla u|^{2})-L_{11}u_{nn}+c_{z}u_{11}+L_{1;1}^{\gamma}u_{\gamma}+L_{1}^{\gamma}L_{1\gamma}u_{n}\no\\
			& \quad~~+c_{x_{1}x_{1}}+c_{x_{1}z}u_{1}+c_{zx_{1}}u_{1}+c_{zz}u_{1}u_{1}+c_{z}L_{11}u_{n}-h_{g;11}+R_{11n}^{i}u_{i}+2L_{11}u_{11}\no\\
			& \quad~~+Bu_{nn}+2u^{\alpha}(\partial_{\alpha}(c-h_{g})+L_{\alpha}^{\gamma}u_{\gamma})+2u_{n}u_{nn}\bigg]\no.
		\end{align}
		
		Since $\sigma_{1}(W)>0$ and $(x_{1},\zeta_{1})$ is a maximum point
		of $G$, for $1+\frac{n(1-t)}{n-2}>0$ we have 
		\[
		u_{nn}+(n-1)u_{11}\ge \Delta u\ge-C
		\]
		and then $u_{nn}\ge-(n-1)u_{11}-C$. We may choose $B=2[\sup_{\mathcal{O}_{1}\cap\partial M}(c(x,u)+|h_{g}|)+n\sup_{\mathcal{O}_{1}\cap\partial M}|L|]$
		such that 
		\[
		(B-L_{11}+2u_{n})u_{nn}\ge-(n-1)(B-L_{11}+2u_{n})u_{11}-(B-L_{11}+2u_{n})C.
		\]
		\vskip 3pt
		Taking $a$ large such that 
		\[
		a>\sup_{\mathcal{O}_{1}\cap\partial M}(n-1)[B+2|L|+2|c_{z}|+2|h_{g}|]
		\]
		and then 
		\[
		u_{11}(x_{1})\le C,
		\]
		where $C$ depends on $\sup_{\mathcal{O}_{1}\cap\partial M}(|\tilde{\nabla}_{xx}^{2}c|+|c_{x_{1}z}|+|c_{zz}|+|c_{z}|+|\tilde{\nabla}_{x}c|+c)$,
		$\|g\|_{C^{2}(\overline{\mathcal{O}_{1}})}$, $\sup_{\mathcal{O}_{1}}|\nabla u|$.
		\vskip 4pt \emph{Case 2:} $x_{1}\in\mathcal{O}_{1}\cap M$. \vskip
		4pt
		
		We now choose geodesic normal coordinates around $x_{1}$ such that
		\[
		g_{ij}=\delta_{ij},\quad\Gamma_{ij}^{k}=0\qquad\mathrm{~~at~~}\:x_{1}.
		\]
		We also use the extended tangent vector field $\zeta_{1}(x)=\frac{1}{\sqrt{g_{11}}}\partial_{x_{1}}$
		near $x_{1}$ such that $\zeta_{1}=\frac{\partial}{\partial x_{1}}$
		at $x_{1}$.
		
		With the above extended vector field $\zeta_{1}$, we rewrite  near $x_{1}$
		\[
		G(x,\zeta_{1})=\eta e^{ad}\left(\frac{u_{11}}{g_{11}}-a^{l}u_{l}+b+B\langle\nabla u,\nabla d\rangle+|\nabla u|^{2}\right),
		\]
		where $a^{l}$ depends on $g,d(x)$, $c(x,u)$ and $\nabla d$,  and $b$ depends on $d, h_g,  c_x$.
		Since $G$ attains its maximum at $(x_{1},\zeta_{1})$ and $\lambda(W)\in\Gamma_{2}^{+}$,
		at $x_{1}$ we have 
		\[
		|W_{ij}|\leq C\sigma_{1}(W)\leq C(\Delta u+1)\leq C(u_{11}+1).
		\]
		This in turn implies that 
		\[
		|u_{ij}|\leq C(u_{11}+1)\qquad\mathrm{~~at~~}x_{1}.
		\]
		For brevity, we let 
		\[
		E:=\frac{u_{11}}{g_{11}}-a^{l}u_{l}+b+B\langle\nabla u,\nabla d\rangle+|\nabla u|^{2}.
		\]
		
		At $x_{1}$, we have 
		\begin{equation}
			0=\partial_{i}\log G=\frac{\eta_{i}}{\eta}+ad_{i}+\frac{\pa_{i}E}{E},\qquad\forall~1\leq i\leq n\label{dG}
		\end{equation}
		and $(\partial_{j}\partial_{i}\log G)$ is negative semi-positivity.
		
		In the following, computations are evaluated at $x_{1}$.
		
		A direct computation yields 
		\begin{align*}
			\partial_{j}\partial_{i}(u_{11})=\partial_{j}(u_{11i}+2\Gamma_{i1}^{l}u_{l1})=u_{11ij}+2\partial_{j}\Gamma_{i1}^{l}u_{l1}
		\end{align*}
		and 
		\begin{align}
			\pa_{i}E=\frac{\pa_{i}(u_{11})}{g_{11}}-\frac{u_{11}}{g_{11}^{2}}\pa_{i}(g_{11})-\pa_{i}a^{l}u_{l}-a^{l}\pa_{i}\pa_{l}u+b_{i}+\left(B\langle\nabla u,\nabla d\rangle+|\nabla u|^{2}\right)_{i}.\label{dE}
		\end{align}
		Then we have 
		\begin{align*}
			\pa_{j}\pa_{i}E= & \pa_{j}\pa_{i}(u_{11})-u_{11}\pa_{j}\pa_{i}g_{11}-\pa_{j}\pa_{i}a^{l}u_{l}-a^{l}\pa_{j}\pa_{i}\pa_{l}u-\pa_{i}a^{l}u_{lj}-\pa_{j}a^{l}u_{li}\\
			& +\pa_{j}\pa_{i}b+\left(B\langle\nabla u,\nabla d\rangle+|\nabla u|^{2}\right)_{ij}\\
			= & \pa_{j}\pa_{i}(u_{11})-u_{11}\pa_{j}\pa_{i}g_{11}-\pa_{j}\pa_{i}a^{l}u_{l}-a^{l}(u_{lij}+\pa_{j}\Gamma_{li}^{k}u_{k})-\pa_{i}a^{l}u_{lj}-\pa_{j}a^{l}u_{li}\\
			& +\pa_{j}\pa_{i}b+\left(B\langle\nabla u,\nabla d\rangle+|\nabla u|^{2}\right)_{ij}\\
			= & u_{11ij}+2\partial_{j}\Gamma_{i1}^{l}u_{l1}-u_{11}\pa_{j}\pa_{i}g_{11}-\pa_{j}\pa_{i}a^{l}u_{l}-a^{l}(u_{lij}+\pa_{j}\Gamma_{li}^{k}u_{k})-\pa_{i}a^{l}u_{lj}-\pa_{j}a^{l}u_{li}\\
			& +\pa_{j}\pa_{i}b+Bu_{kij}d_{k}+Bu_{ki}d_{kj}+Bu_{kj}d_{ki}+Bu_{k}d_{kij}+2u_{kj}u_{ki}+2u_{k}u_{kij}.
		\end{align*}
		
		It follows from \eqref{dG} and \eqref{dE} that 
		\begin{equation}
			u_{11i}=-E(\frac{\eta_{i}}{\eta}+ad_{i})-Bu_{ki}d_{k}+a^{l}u_{li}+\pa_{i}a^{l}u_{l}-b_{i}-Bu_{k}d_{ki}-2u_{k}u_{ki}.\label{third_derivative_u}
		\end{equation}
		Thus, again by \eqref{dG} we obtain 
		\begin{align*}
			& \partial_{j}\partial_{i}\log G\\
			= & \frac{\eta_{ij}}{\eta}-\frac{\eta_{i}\eta_{j}}{\eta^{2}}+ad_{ij}+\frac{\partial_{j}\partial_{i}E}{E}-\frac{\partial_{i}E\partial_{j}E}{E^{2}}\\
			= & \frac{\eta_{ij}}{\eta}-\frac{\eta_{i}\eta_{j}}{\eta^{2}}+ad_{ij}+\frac{\partial_{j}\partial_{i}E}{E}-(\frac{\eta_{i}}{\eta}+ad_{i})(\frac{\eta_{j}}{\eta}+ad_{j}).
		\end{align*}		
		Differentiating the equation, we have 
		\begin{equation}
			\tilde{P}^{ij}u_{ijk}+\tilde{F}^{ij}\bigg(-\frac{2-t}{2}2u_{p}u_{pk}g_{ij}+u_{ik}u_{j}+u_{jk}u_{i}\bigg)=-\tilde{F}^{ij}W_{ij,k}-f_{k}\label{eq:first derivative to the equation}
		\end{equation}
		and 
		\begin{align}
			&  \tilde{P}^{ij}u_{ijkk}+\tilde{F}^{ij}\bigg(-\frac{2-t}{2}(|\nabla u|^{2})_{kk}g_{ij}+u_{ikk}u_{j}+u_{jkk}u_{i}+2u_{ik}u_{jk}\bigg)\label{eq:second derivative to equation}\\
			& =f_{kk}-\tilde{F}^{ij,rs}W_{ij,k}W_{rs,k}-\tilde{F}^{ij}W_{ij,kk}.\nonumber 
		\end{align}
		
		Hence, putting these facts together we conclude that 
		\begin{align*}
			0\ge & E\tilde{P}^{ij}\pa_{i}\pa_{j}(\log G)\\
			= & E\tilde{P}^{ij}(\frac{\eta_{ij}}{\eta}-2\frac{\eta_{i}\eta_{j}}{\eta^{2}})+aE\tilde{P}^{ij}d_{ij}-a^{2}E\tilde{P}^{ij}d_{i}d_{j}-2aE\tilde{P}^{ij}\frac{\eta_{i}}{\eta}d_{j}+\tilde{P}^{ij}\partial_{j}\partial_{i}E\\
			\geq & E\tilde{P}^{ij}(\frac{\eta_{ij}}{\eta}-3\frac{\eta_{i}\eta_{j}}{\eta^{2}})-CE\sum \tilde{F}^{ii}+\tilde{P}^{ij}\left(u_{11ij}+Bu_{kij}d_{k}-a^{l}u_{lij}+2u_{k}u_{kij}\right)\\
			& +2B\tilde{P}^{ij}u_{ki}d_{kj}+B\tilde{P}^{ij}u_{k}d_{kij}+2\tilde{P}^{ij}u_{kj}u_{ki}+\tilde{P}^{ij}\pa_{j}\pa_{i}b-2\tilde{P}^{ij}\pa_{i}a^{l}u_{lj}-\tilde{P}^{ij}\pa_{j}\pa_{i}a^{l}u_{l}.
		\end{align*}
		
		Using 
		\begin{align*}
			& \tilde{P}^{ij}u_{11ij}\\
			= & \tilde{P}^{ij}(u_{ij11}+R_{11i}^{m}u_{mj}+2R_{i1j}^{m}u_{m1}+R_{11j}^{m}u_{mi}+R_{11i,j}^{m}u_{m}+R_{i1j,1}^{m}u_{m})\\
			\geq & \tilde{P}^{ij}u_{ij11}+2\tilde{P}^{ij}R_{11i}^{m}u_{mj}+2\tilde{P}^{ij}R_{i1j}^{m}u_{m1}-C\sum_{i}\tilde{F}^{ii},
		\end{align*}
		we have 
		\begin{align}
			0\ge & E\tilde{P}^{ij}\pa_{j}\pa_{i}(\log G)\label{eq:main negative}\\
			\ge & E\tilde{P}^{ij}(\frac{\eta_{ij}}{\eta}-3\frac{\eta_{i}\eta_{j}}{\eta^{2}})+2\tilde{P}^{ij}u_{kj}u_{ki}+\tilde{P}^{ij}u_{ij11}+2\tilde{P}^{ij}u_{k}u_{kij}+B\tilde{P}^{ij}u_{kij}d_{k}-\tilde{P}^{ij}a^{l}u_{lij}\nonumber \\
			& -CE\sum \tilde{F}^{ii}-C\sum_{i}\tilde{F}^{ii}\nonumber \\
			:= & E\tilde{P}^{ij}(\frac{\eta_{ij}}{\eta}-3\frac{\eta_{i}\eta_{j}}{\eta^{2}})+2\tilde{P}^{ij}u_{kj}u_{ki}+I+II+III-CE\sum \tilde{F}^{ii}-C\sum_{i}\tilde{F}^{ii}\nonumber 
		\end{align}
		where $I=\tilde{P}^{ij}u_{ij11},II=2\tilde{P}^{ij}u_{k}u_{kij},III=(Bd_{k}-a_{k})\tilde{P}^{ij}u_{kij}$.
		
		By (\ref{eq:second derivative to equation}) and  $-\tilde{F}^{ij,rs}W_{ij,k}W_{rs,k}\ge0$,
		we have 
		\begin{align}\label{term-I}
			I\geq & f_{11}+\tilde{F}^{ij}(-2u_{i11}u_{j}+(2-t)u_{k}u_{k11}g_{ij})+\tilde{F}^{ij}(-2u_{1i}u_{1j}+(2-t)u_{k1}^{2}g_{ij})-C\sum_{i}\tilde{F}^{ii}.
		\end{align}
		Using 
		\[
		u_{i11}=u_{1i1}=u_{11i}+R_{1i1}^{m}u_{m}
		\]
		and \eqref{third_derivative_u}, we have 
		\begin{equation}
			\begin{aligned} & \tilde{F}^{ij}(-2u_{i11}u_{j}+(2-t)u_{k}u_{k11}g_{ij})\\
				= & -2u_{j}\tilde{F}^{ij}\big[-E(\frac{\eta_{i}}{\eta}+ad_{i})-B(u_{li}d_{l}+d_{li}u_{l})+a^{l}u_{li}+\pa_{i}a^{l}u_{l}-b_{i}-2u_{l}u_{li}+R_{1i1}^{m}u_{m}\big]\\
				& +(2-t)u_{k}(\sum_{i}\tilde{F}^{ii})\big[-E(\frac{\eta_{k}}{\eta}+ad_{k})-B(u_{lk}d_{l}+d_{lk}u_{l})+a^{l}u_{lk}+\pa_{k}a^{l}u_{l}-b_{k}-2u_{l}u_{lk}+R_{1k1}^{m}u_{m}\big]\\
				\geq & 4u_{j}\tilde{F}^{ij}u_{k}u_{ki}+2Bu_{j}\tilde{F}^{ij}(u_{li}d_{l}+d_{li}u_{l})-2u_{j}\tilde{F}^{ij}(a^{l}u_{li}+\pa_{i}a^{l}u_{l})-2u_{j}\tilde{F}^{ij}R_{1i1}^{m}u_{m}\\
				& -2(2-t)(\sum_{i}\tilde{F}^{ii})u_{k}u_{l}u_{lk}-(2-t)B(\sum_{i}\tilde{F}^{ii})u_{k}(u_{lk}d_{l}+d_{lk}u_{l})+(2-t)u_{k}(\sum_{i}\tilde{F}^{ii})(a^{l}u_{lk}+\pa_{k}a^{l}u_{l})\\
				& +2Eu_{j}\tilde{F}^{ij}(\frac{\eta_{i}}{\eta}+ad_{i})-(2-t)E(\sum_{i}\tilde{F}^{ii})u_{k}(\frac{\eta_{k}}{\eta}+ad_{k})+(\sum_{i}\tilde{F}^{ii})u_{k}R_{1k1}^{m}u_{m}-C\sum_{i}\tilde{F}^{ii}.
			\end{aligned}
			\label{eq:third deritive term}
		\end{equation}
		Using 
		\[
		\begin{aligned}\tilde{F}^{ij}u_{lij} & =\tilde{F}^{ij}u_{ijl}+\tilde{F}^{ij}R_{ilj}^{m}u_{m}\\
			& \ge \tilde{F}^{ij}u_{ijl}-C\sum_{i}\tilde{F}^{ii}
		\end{aligned}
		\]
		and (\ref{eq:first derivative to the equation}), we obtain
		
		\begin{align}
			II:= & 2\tilde{P}^{ij}u_{l}u_{lij}\no\label{eq:II}\\
			= & 2u_{l}f_{l}-2u_{l}\tilde{F}^{ij}(u_{i}u_{jl}+u_{j}u_{il}-(2-t)u_{k}u_{kl}g_{ij}+A_{ij,l})-C\sum_{i}\tilde{F}^{ii}
		\end{align}
		and 
		\begin{align}
			III:= & (Bd_{k}-a_{k})\tilde{P}^{ij}u_{kij}\no\label{eq:III}\\
			\ge & (Bd_{k}-a_{k})f_{k}-C\sum_{i}\tilde{F}^{ii}-(Bd_{k}-a_{k})\tilde{F}^{ij}(u_{i}u_{jk}+u_{j}u_{ik}-(2-t)u_{l}u_{lk}g_{ij}+A_{ij,k}).
		\end{align}
		
		Therefore, we combine (\ref{eq:main negative})(\ref{eq:third deritive term})(\ref{term-I})(\ref{eq:II})
		and (\ref{eq:III}) to conclude that
		
		\begin{align*}
			0\ge & E\tilde{P}^{ij}(\log G)_{ij}\\
			\ge & E\tilde{P}^{ij}(\frac{\eta_{ij}}{\eta}-3\frac{\eta_{i}\eta_{j}}{\eta^{2}})+2\tilde{P}^{ij}u_{kj}u_{ki}+\tilde{F}^{ij}(-2u_{1i}u_{1j}+(2-t)u_{k1}^{2}g_{ij})\\
			& -C(\sum_{i}\tilde{F}^{ii})\frac{|\nabla\eta|}{\eta}E-CE\sum \tilde{F}^{ii}-C\sum_{i}\tilde{F}^{ii}-\|f\|_{C^{2}(\overline{\mathcal{O}_{1}})}\\
			\ge & (2-t)(\sum_{i}\tilde{F}^{ii})u_{k1}^{2}+E\tilde{P}^{ij}(\frac{\eta_{ij}}{\eta}-3\frac{\eta_{i}\eta_{j}}{\eta^{2}})-C(\sum_{i}\tilde{F}^{ii})\frac{|\nabla\eta|}{\eta}E-CE\sum \tilde{F}^{ii}\\
			& -C\sum_{i}\tilde{F}^{ii}-\|f\|_{C^{2}(\overline{\mathcal{O}_{1}})}.
		\end{align*}
		This implies 
		\[
		u_{11}(x_{1})\le C.
		\]
		
		Combining the above two cases we obtain the desired estimate. 
	\end{proof}
	It remains to estimate the double normal derivative of $u$ on the
	boundary. Define 
	\[
	M_{r}^{*}:=\{x\in\overline{M};d_{g}(x,\partial M)<r\}.
	\]
	\subsection{The double normal derivative on $\partial M$}\label{Subsect2.2}
	\begin{thm}\label{high dimen-double normal}
		Let $u$ be a $C^3$ solution to (\ref{eq:general equation}) for $t\le 1$. Then 
		\[
		\sup_{\partial M}u_{nn}\le C,
		\]
		where $C$ is a positive constant depending on $g$,	$\sup_{M_{r}^{*}}|\nabla u|$,$$\sup_{\partial M}(|\tilde{\nabla}_{xx}^{2}c|+|c_{xz}|+|c_{zz}|+|c_{z}|+|\tilde{\nabla}_{x}c|+c),$$
		and 
		\[
		\sup_{M_{r}^{*}}(\frac{1}{f(x,u)}+f(x,u)+|f_{x}|+|f_{z}|+|f_{xx}|+|f_{xz}|+|f_{xz}|).
		\]
	\end{thm}
	
	\begin{proof}
		Without loss of generality, we assume that 
		\[
		\sup_{\partial M}u_{nn}=u_{nn}(x_{0}):=M^{\ast}\qquad\mathrm{for~~some~~}\,x_{0}\in\partial M.
		\]
		In a tubular neighborhood near $\partial M$ we define 
		\begin{align*}
			G(x)= & \langle\nabla u,\nabla d_{g}(x,\partial M)\rangle-c(x,u)+h_{g}+\left(\langle\nabla u,\nabla d_{g}(x,\partial M)\rangle-c(x,u)+h_{g}\right)^{7/5}\\
			& -\frac{1}{2}Nd_{g}(x,\partial M)-N_{1}d_{g}^{2}(x_{0},x),
		\end{align*}
		where $N=N_{2}+M^{*}$, $N_{1},N_{2}\in\mathbb{R}_{+}$ are to be
		determined later and $N_{2}>N_{1}$. Let $x'\in \partial M$ is the point closest to $x$.  Here, $c(x,u)$ is the natural extension of $c(x,u)$ on $\partial M$ such that $c(x,u)=c(x',u(x))$,   and $h_{g}(x)$ is extended in the way that $h_{g}(x)=h_{g}(x')$. 
		
		It suffices to consider $G$ in $M_{r}:=\{x\in\overline{M};d_{g}(x,x_{0})<r\}$.
		Our goal is to show that with appropriate choice of $N_{1}$ and $N_{2}$,
		$x_{0}$ is a maximum point of $G$ in $\overline{M_{r}}$. If we
		admit this claim temporarily, then under Fermi coordinates around
		$x_{0}$ with $\nu=\pa_{x_{n}}$ we have 
		\begin{align*}
			0 & \ge G_{n}(x_{0})=(u_{nn}-\varphi_{n})(x_{0})[1+\frac{7}{5}(u_{n}-\varphi)^{2/5}(x_{0})]-\frac{1}{2}N\\
			& =(u_{nn}-\varphi_{n})(x_{0})-\frac{1}{2}N\\
			& =\frac{1}{2}u_{nn}(x_{0})-\varphi_{n}(x_{0})-\frac{1}{2}N_{2}.
		\end{align*}
		This gives $u_{nn}(x_{0})\le N_{2}+2\|\varphi\|_{C^{1}(\overline{M_{r}})}$.
		
		By contradiction, we assume 
		\[
		\max_{\overline{M_{r}}}G=G(x_{1})>0\quad\mathrm{~~for~~}\quad x_{1}\in\mathring{M}_{r}.
		\]
		
		We choose Fermi coordinates around $x_{0}$ such that 
		\[
		g=\ud x_{n}^{2}+g_{\alpha\beta}\ud x_{\alpha}\ud x_{\beta}\qquad\mathrm{~~in~~}B_{\rho}^{+}
		\]
		with $\nu=\pa_{x_{n}}$, meanwhile $W^{\top}=\big(W_{\beta}^{\alpha}\big)$
		is diagonal at $x_{1}$. For $x$ near $\partial M$, we let $d_{g}(x,x')=d_{g}(x,\partial M)$
		for some $x'\in\partial M$ and extend $c(x)=c(x'),h_{g}(x)=h_{g}(x')$
		for $x=(x',x_{n})\in B_{\rho}^{+}$. Moreover, there holds 
		\[
		(x_{n})_{,\alpha\beta}=-L_{\alpha\beta}(x')+O(x_{n})\qquad\mathrm{~~for~~}x\in B_{\rho}^{+}.
		\]
		
		For sufficiently small $r$, we choose $N_{1}$ such that 
		\begin{align}
			N_{1}r^{2}>\sup_{M_{r}}|\langle\nabla u,\nabla d_{g}(x,\partial M)\rangle-c(x,u)+h_{g}+(\langle\nabla u,\nabla d_{g}(x,\partial M)\rangle-c(x,u)+h_{g})^{7/5}|.\label{eq:condtion 0-1}
		\end{align}
		It is not hard to check that $G<0$ on $\{d_{g}(x,x_{0})=r\}\cap M$;
		$G(x_{0})=0$ and $G<0$ on $(\partial M\cap M_{r})\setminus\{x_{0}\}$.
		
		Fix $r$ such that $0<r<\rho$ and let $\eta(x)=d_{g}^{2}(x,x_{0})$,
		then there exists a positive constants $\delta=\delta(r)$ and $\mu=\mu(r)\leq r$
		such that 
		\[
		\|\eta\|_{C^{2}(\overline{M_{r}})}\leq\delta
		\]
		and $0\leq x_{n}\leq\mu$ for all $x\in M_{r}$.
		
		Until now the constant $N_{1}$ has been fixed, it remains to choose
		a sufficiently large $N$ depending on $r,N_{1},g$, $\sup_{M_{r}}(|c|+|\nabla u|+e^{-u})$,
		$\sup_{M_{r}}(f(x,u)+|f_{x}|+|f_{z}|)$ and $\sup_{M_{r}}\frac{1}{f(x,u)}$
		such that the maximum point of $G$ is $x_{0}$. Without confusion,
		the following constant $C$ may be different line by line and only
		depends on the above known data independent of $N$.
		
		It follows from Theorem \ref{Thm:reduction} that for each $1\le\alpha\le n-1$,
		\[
		u_{\alpha}^{\alpha}\le M^{*}+C
		\]
		 together with $\sigma_{1}(W)>0$, that is,
		\[
		-C\le u_{1}^{1}+\cdots+u_{n-1}^{n-1}+u_{nn}\le u_{\alpha}^{\alpha}+(n-1)M^{*}+C
		\]
		implying
		\[
		-u_{\alpha}^{\alpha}\le(n-1)M^{*}+C.
		\]
		Thus, we combine these two estimates to show
		\begin{equation}
			|u_{\alpha}^{\alpha}|\le(n-1)M^{*}+C\le(n-1)N+C.\label{eq:tangential control}
		\end{equation}
		
		Notice that $G(x_{1})>0$ implies 
		\begin{align*}
			& \big(u_{n}(x_{1})-c(x_1,u)+h_{g}(x_{1}')\big)(1+(u_{n}(x_{1})-c(x_1,u)+h_{g}(x_{1}'))^{2/5})\\
			> & \frac{1}{2}Nd_{g}(x_{1},\partial M)+N_{1}d_{g}^{2}(x_{0},x_{1})>0,
		\end{align*}
		which yields that $A:=u_{n}(x_{1})-c(x_1,u)+h_{g}(x_{1}')>0$.
		More precisely, we have 
		\begin{equation}
			A\ge\frac{\frac{1}{2}Nx_{n}+N_{1}d_{g}^{2}(x_{0},x_{1})}{1+A_{1}^{2/5}},\label{eq:lower bound of A}
		\end{equation}
		where $A_{1}=\sup_{M_{r}}(|\nabla u|+|h_{g}|+c)$.
		
		Near $x_{1}$ we have 
		\[
		G=u_{n}-c(x,u)+h_{g}+\big(u_{n}-c(x,u)+h_{g}\big)^{7/5}-\frac{1}{2}Nx_{n}-N_{1}d^{2}(x_{0},x),
		\]
		then for all $1\le i\le n$, at $x_{1}$ 
		\begin{align*}
			0= & G_{i}\\
			= & \big(u_{ni}+u^{\alpha}(x_{n})_{\alpha i}-(c-h_{g})_{i}\big)(1+\frac{7}{5}\big(u_{n}-c+h_{g}\big)^{\frac{2}{5}})-\frac{1}{2}N\delta_{ni}-2N_{1}d_{g}(x_{0},x)d_{g}(x_{0},x)_{i}.
		\end{align*}
		In particular, for large sufficiently $N$,  we have 
		\begin{equation}
			u_{nn}=\frac{\frac{1}{2}N+2N_{1}d_{g}(x_{0},x)d_{g}(x_{0},x)_{n}}{1+\frac{7}{5}A^{\frac{2}{5}}}+\varphi_{n}\ge\frac{\frac{1}{4}N}{1+\frac{7}{5}A^{\frac{2}{5}}},\label{lbd:u_nn}
		\end{equation}
		and 
		\begin{equation}
			u_{n\alpha}=-u^{\beta}(x_{n})_{\beta\alpha}+\frac{2N_{1}d_{g}(x_{0},x)d_{g}(x_{0},x)_{\alpha}}{1+\frac{7}{5}A^{\frac{2}{5}}}+\varphi_{\alpha}.\label{u_n alpha}
		\end{equation}
		
		In the following computations, we find it convenient to introduce: 
		\begin{equation}\label{new_def_F_i^j}
			F_{j}^{i}:=\frac{\pa\sigma_{2}(W)}{\pa W_{i}^{j}},\quad F_{j}^{i}=g^{ik}F_{kj},\, \mathrm{and}\quad	P^{i}_j=F^{i}_j+\frac{1-t}{n-2}\sum_{k=1}^{n}F^{k}_k\delta^{i}_j.
		\end{equation}
		
		Since $\lambda(W)\in\Gamma_{2}^{+}$, it follows from Theorem \ref{Thm:reduction},
		the choice of $N$ and (\ref{lbd:u_nn}) that 
		\begin{equation}
			|W_{\beta}^{\alpha}(x_{1})|\le\sigma_{1}(W)(x_{1})\leq Cu_{nn}(x_{0})\le Cu_{nn}(x_{1})\leq CW_{n}^{n}(x_{1})\le CN.\label{eq:doubel tangential initial bound}
		\end{equation}
		This directly implies that 
		\[
		\sum_{i}F_{i}^{i}\leq CW_{n}^{n}\qquad\mathrm{~~at~~}x_{1}.
		\]
		
		As $(G_{ij})$ is semi-negative definite at $x_{1}$,  we have
		\begin{align}
			0\ge & P_{j}^{i}G_{i}^{j}\label{eq:matrix non-negative}\\
			\ge & \frac{14}{25}A^{-\frac{3}{5}}P_{j}^{i}(u_{ni}+u^{k}(x_{n})_{ki}-\varphi_{i})(u_{n}^{j}+u^{l}(x_{n})_{l}^{j}-\varphi^{j})\no\\
			& +(1+\frac{7}{5}A^{\frac{2}{5}})P_{j}^{i}\left(u^{\beta j}(x_{n})_{\beta i}+u^{k}(x_{n})_{ki}^{j}+u_{i}^{\beta}(x_{n})_{\beta}^{j}\right)\no\\
			& +P_{j}^{i}(u_{ni}^{\quad j}-\varphi_{i}^{j})(1+\frac{7}{5}A^{\frac{2}{5}})-\frac{N}{2}P_{j}^{i}(x_{n})_{i}^{j}-\delta N_{1}\sum_{i}F_{i}^{i}.\no
		\end{align}
		In the following, we use  $O(1)$  to denote the quantity depending on $|u|_{C^1}$, $g$.
		
		By the definition of $\sigma_2(W)$,  
		\begin{align}\label{F_alphan}
			F_{\alpha}^{n}=-W_{\alpha}^{n}=-(u_{n\alpha}+u_{n}u_{\alpha}+A_{n\alpha}^{t})=O(1)\quad \text{and}\quad P_{\alpha}^{n}=O(1),
		\end{align}
		\begin{equation}\label{F_nn}
			F_{n}^{n}=\sigma_{1}(W^{\top})=\sum_{\alpha}W_{\alpha}^{\alpha}=\frac{f^{2}-\sigma_{2}(W^{\top})+\sum_{\alpha}W_{n}^{\alpha}W_{\alpha}^{n}}{W_{n}^{n}},
		\end{equation}
		and 
		\begin{equation}\label{F_betaalpha}
			F_{\beta}^{\alpha}=\sigma_1(W)\delta_{\beta}^{\alpha}-W^{\alpha}_{\beta}.
		\end{equation}
		Observe that 
		\begin{align*}
			W_{nn} & =(1-t)\frac{\Delta u}{n-2}+u_{nn}+O(1)\\
			& =\frac{\frac{1}{2}N+2N_{1}d(x_{0},x_{1})d(x_{0},x)_{n}}{1+\frac{7}{5}A^{\frac{2}{5}}}+(1-t)\frac{\Delta u}{n-2}+O(1),
		\end{align*}
		and		
		\[
		\sigma_{1}(W)=(1+\frac{n(1-t)}{n-2})\Delta u+O(1)
		\]
		
		By
		\begin{align}
			& \sum_{i}F_{i}^{i}=(n-1)\sigma_{1}(W)\label{bdd_trace_F-1-1}\\
			> & (n-1)W_{nn}\no\\
			\geq & (n-1)u_{nn}+(1-t)\frac{(n-1)\Delta u}{n-2}-C\no\\
			\ge & \frac{\frac{1}{4}N(n-2)}{1+\frac{7}{5}A^{\frac{2}{5}}},\no
		\end{align}		
		and the  estimates \eqref{lbd:u_nn}\eqref{u_n alpha}\eqref{F_alphan}\eqref{F_betaalpha}\eqref{F_nn}, we have
		\begin{equation}\label{est:important-1}
			|F_{\beta}^{\alpha}u_{n}^{\beta}|+|F_{n}^{\alpha}u_{n}^{n}|+|F_{\alpha}^{n}u_{n}^{\alpha}|\le  C\sum_{i}F_{i}^{i}.
		\end{equation}
		%
		
		Differentiating the equation as (\ref{eq:first derivative to the equation}), combined with \eqref{est:important-1},  we obtain 
		\begin{align}
			& P_{j}^{i}(u_{ni}^{\quad j}-\varphi_{i}^{j})\label{eq:third derivative}\\
			= & P_{j}^{i}u_{in}^{j}+P_{j}^{i}R_{nil}^{j}u^{l}-P_{j}^{i}\varphi_{i}^{j}\nonumber \\
			= & 2ff_{n}-2F_{n}^{n}u_{n}u_{n}^{n}-2F_{n}^{\alpha}u_{\alpha}u_{n}^{n}-2F_{\alpha}^{n}u_{n}u_{n}^{\alpha}-2F_{\beta}^{\alpha}u_{\alpha}u_{n}^{\beta}\nonumber \\
			& +(2-t)(\sum_{i}{F}_{i}^{i})u^{n}u_{nn}+(2-t)(\sum_{i}{F}_{i}^{i})u^{\alpha}u_{\alpha n}-F_{j}^{i}A_{i,n}^{j}+F_{j}^{i}R_{nil}^{j}u^{l}\nonumber \\
			& -P_{j}^{i}(c_{x_{i}}^{j}+c_{x_{i}z}u^{j}+c_{z}^{j}u_{i}+c_{zz}u_{i}u^{j}-(h_{g})_{i}^{j})\nonumber \\
			& -c_zP_{j}^{i}(W_{i}^{j}-\frac{1-t}{n-2}\Delta u \delta_i^j-u_{i}u^{j}+\frac{2-t}{2}|\nabla u|^{2}\delta_{i}^{j}-A_{i}^{j})e^{-u}\nonumber \\
			\ge & \bigg((n-1)(2-t)W_{n}^{n}+[(n-3)+(1-t)(n-1)]\sigma_{1}(W^{\top})\bigg)u_{n}u_{nn}\no\\
			&+c_zP_{i}^{i}\frac{1-t}{n-2}\Delta u e^{-u}-C\sum_{i}F_{i}^{i},\nonumber 
		\end{align}
		where 	\[
		\varphi_{i}^{j}=c(x,u)_{i}^{j}=(c_{x_{i}}+c_{z}u_{i})^{j}=c_{x_{i}}^{j}+c_{x_{i}z}u^{j}+c_{z}^{j}u_{i}+c_{zz}u_{i}u^{j}+c_{z}u_{i}^{j}-(h_{g})_{i}^{j}.
		\]
		By (\ref{lbd:u_nn}) we have 
		\begin{align}
			& (2-t)(n-1)W_{n}^{n}u_{n}u_{nn}(1+\frac{7}{5}A^{\frac{2}{5}})\no\\
			= & (2-t)(n-1)W_{n}^{n}u_{n}\bigg(\frac{1}{2}N+2N_{1}d_{g}(x_{0},x)d_{g}(x_{0},x)_{n}+(1+\frac{7}{5}A^{2/5})\varphi_{n}\bigg)\no\\
			\ge & (2-t)\frac{n-1}{2}W_{n}^{n}Nu_{n}-CW_{n}^{n}.\label{key5-1}
		\end{align}
		
		Also	\begin{align}
			&-\frac{N}{2}P_{j}^{i}(x_{n})_{i}^{j}\\
			= & -\frac{N}{2}P_{\beta}^{\alpha}(x_{n})_{\alpha}^{\beta}=\frac{N}{2}P_{\beta}^{\alpha}(L_{\alpha}^{\beta}(x_1')+O(x_{n}))\no\label{eq:key4-1}\\
			= & \frac{N}{2}(W_{nn}+\sigma_{1}(W^{\top}))\cdot(n-1)h_{g}+\frac{N}{2}\frac{1-t}{n-2}F_{k}^{k}(n-1)h_{g}-\frac{N}{2}W_{\beta}^{\alpha}L_{\alpha}^{\beta}-C(\sum_{i}F_{i}^{i})Nx_{n}.
		\end{align}
		And then		
		\begin{equation}\label{non-negative c}
			\begin{aligned}
				& (2-t)(n-1)W_{n}^{n}u_{n}u_{nn}(1+\frac{7}{5}A^{\frac{2}{5}})+ \frac{N}{2}(W_{nn}+\sigma_{1}(W^{\top}))\cdot(n-1)h_{g}\\
				\ge &\frac{N(n-1)}{2}W_{n}^{n}(u_{n}+h_{g})+ (1-t)\frac{n-1}{2}W_{n}^{n}Nu_{n}+\frac{N}{2}\sigma_{1}(W^{\top})\cdot(n-1)h_{g}-CW^n_n\\
				\ge & (A+c(x_{1},u))\frac{N(n-1)}{2}W_{n}^{n}+ (1-t)\frac{n-1}{2}W_{n}^{n}Nu_{n}+\frac{N}{2}\sigma_{1}(W^{\top})\cdot(n-1)h_{g}-CW^n_n.
			\end{aligned}
		\end{equation}	
		By (\ref{eq:key4-1})(\ref{eq:third derivative})(\ref{non-negative c}) we have
		\begin{align}
			& -\frac{N}{2}P_{j}^{i}(x_{n})_{i}^{j}+(1+\frac{7}{5}A^{\frac{2}{5}})P_{j}^{i}(u_{ni}^{\quad j}-\varphi_{i}^{j})\no\\
			\ge & (1-t)\frac{n-1}{2}W_{n}^{n}Nu_{n}+\frac{N}{2}\sigma_{1}(W^{\top})\cdot(n-1)h_{g}+\frac{AN(n-2)}{2}W^n_n-\frac{N}{2}W_{\beta}^{\alpha}L_{\alpha}^{\beta}\label{eq:Neumann boundary condition}\\
			& +[(n-3)+(1-t)(n-1)]\sigma_{1}(W^{\top})u_{n}u_{nn}(1+\frac{7}{5}A^{\frac{2}{5}})+\frac{N}{2}\frac{1-t}{n-2}F_{k}^{k}(n-1)h_{g}\no\\
			&+c_z(1+\frac{7}{5}A^{\frac{2}{5}})\frac{1-t}{n-2}\Delta u e^{-u}P_{i}^{i}-C\sum_{i}F_{i}^{i}\nonumber 
		\end{align}
		where \eqref{eq:Neumann boundary condition} holds due to $c(x_1,u)\ge0$ and $A\frac{N}{2}W^n_n-C(\sum_{i}F_{i}^{i})Nx_{n}>0$ for (\ref{eq:lower bound of A}). 
		
		By \eqref{eq:Neumann boundary condition}\eqref{eq:tangential control} we have
		\begin{equation}\label{eq:for sigma_1(W^T) big use}
			-\frac{N}{2}P_{j}^{i}(x_{n})_{i}^{j}+(1+\frac{7}{5}A^{\frac{2}{5}})P_{j}^{i}(u_{ni}^{\quad j}-\varphi_{i}^{j})\ge -CW_{nn}^{2}.
		\end{equation}

	Denote $\sigma_{1}(W|W_{\alpha}^{\alpha})=\sigma_{1}(W)-W_{\alpha}^{\alpha}.$
			
			Also as
		\begin{align*}
			P_{j}^{\gamma}u^{\beta j}L_{\beta\gamma} & =P_{n}^{\gamma}u^{\beta n}L_{\beta\gamma}+P_{\zeta}^{\gamma}u^{\beta\zeta}L_{\beta\gamma}\\
			& =P_{n}^{\gamma}u^{\beta n}L_{\beta\gamma}+(\sigma_{1}(W)\delta_{\zeta}^{\gamma}-W_{\zeta}^{\gamma})u^{\beta\zeta}L_{\beta\gamma}+\frac{1-t}{(n-2)}F_{k}^{k}u_{\gamma}^{\beta}L_{\beta}^{\gamma}\\
			& =O(N_{1}^{2})+\sigma_{1}(W|W_{\alpha}^{\alpha})u_{\alpha}^{\beta}L_{\beta}^{\alpha}+\frac{1-t}{(n-2)}F_{k}^{k}u_{\gamma}^{\beta}L_{\beta}^{\gamma}
		\end{align*}
		and then
		\begin{align}\label{extra boundary term}
			& (1+\frac{7}{5}A^{\frac{2}{5}})P_{j}^{i}\left(u^{\beta j}(x_{n})_{\beta i}+u^{k}(x_{n})_{ki}^{j}+u_{i}^{\beta}(x_{n})_{\beta}^{j}\right)\\
			\ge & -2(1+\frac{7}{5}A^{\frac{2}{5}})\sigma_{1}(W|W_{\alpha}^{\alpha})u_{\alpha}^{\beta}L_{\beta}^{\alpha}\nonumber\\
			&-\frac{2(1-t)}{(n-2)}(1+\frac{7}{5}A^{\frac{2}{5}})F_{k}^{k}u_{\gamma}^{\beta}L_{\beta}^{\gamma}-CF_{k}^{k}-CN_{1}^{2}.\no
		\end{align}
		\vskip 4pt
		
		We next present two claims, postponing their proof to the end of proof of this theorem.\\

		\noindent\fbox{\begin{minipage}[t]{1\columnwidth-2\fboxsep-2\fboxrule}
				\begin{claim}[1]	
					\begin{align*}
						& P_{j}^{i}(u_{ni}+u^{k}(x_{n})_{ki}-\varphi_{i})(u_{n}^{j}+u^{l}(x_{n})_{l}^{j}-\varphi^{j})\no\\
						\ge & \frac{1-t}{n-2}F_{k}^{k}u_{ni}u_{n}^{i}+u_{nn}\sigma_{1}(W^{\top})u_{nn}-CW_{nn}^{2}.
					\end{align*}
				\end{claim}
		\end{minipage}}
		$\vspace{2pt}$
		
		The proof of Claim (1) relies on \eqref{eq:tangential control}\label{eq:for sigma_1(W^T) big use} and can be employed to obtain an upper bound for $\sigma_{1}(W^{\top})(x_1)$. 
			By \eqref{eq:matrix non-negative}(\ref{eq:for large sigma_1(W^T) large use})(\ref{eq:tangential control}) and Claim (1),
		we also have the following estimate:
		\begin{align}
			0 & \ge P_{j}^{i}G_{i}^{j}\label{eq:sigma_1(W^T) large}\\
			& \ge \frac{14}{25}A^{-\frac{3}{5}}\frac{1-t}{(n-2)}F_{k}^{k}u_{ni}u_{n}^{i}+ \frac{14}{25}A^{-\frac{3}{5}}\left(u_{nn}\sigma_{1}(W^{\top})u_{nn}-Cu_{nn}^{2}\right)-Cu_{nn}^2\nonumber \\
			& \ge\big( \frac{14}{25}A_1^{-\frac{3}{5}}\sigma_{1}(W^{\top})-C_{3}\big)u_{nn}^{2},\no
		\end{align}
		where $C_{3}$ is a positive constant depending on $g$, $c$, $\sup_{M_{r}}e^{-u}$
		and $\|\nabla u\|_{C^{0}(\overline{M_{r}})}.$
		
		This directly yields
		\begin{equation}\label{upper bound of C3}
			\sigma_{1}(W^{\top})(x_1)\le C_{3}.
		\end{equation}
		
		This upper bound serves as an additional algebraic condition for later discussion. However, Claim (1) alone is insufficient for $t=1$. To this end, we need the following 
		\medskip

		\noindent\fbox{\begin{minipage}[t]{1\columnwidth-2\fboxsep-2\fboxrule}
				\begin{claim}[2]
					\begin{align}\label{eq:key estimate}
						&P_{j}^{i}(u_{ni}+u^{k}(x_{n})_{ki}-\varphi_{i})(u_{n}^{j}+u^{l}(x_{n})_{l}^{j}-\varphi^{j})\\
						\ge & u_{nn}\bigg(f-\sigma_{2}(W^{\top})-C\sigma_{1}(W^{\top})\bigg)\nonumber \\
						& +\frac{1-t}{2(n-2)}F_{k}^{k}u_{ni}u_{n}^{i}-\frac{(1-t)}{n-2}\Delta u\sigma_{1}(W^{\top})u_{nn}-C(1-t)F_{k}^{k}-C\sum_{\alpha}|W_{\alpha}^{\alpha}|-C.\nonumber 
					\end{align}
				\end{claim}
		\end{minipage}}
		\vskip 4pt

		For any fixed positive constant $L$, taking $N$ large such that 
		\begin{align}
			&\frac{1}{L(n-2)}F_{k}^{k}u_{ni}u_{n}^{i}\label{the chioce of N}\\
			\ge& 2(A_{1}^{^{3/5}}+1)\bigg\{  \frac{n-1}{2}|W_{n}^{n}Nu_{n}|+(2-t)\frac{(n-1)}{2(n-2)}|\Delta uu_{n}N|r\no\\
			&+(n-1)|\sigma_{1}(W^{\top})u_{n}u_{nn}|
			 +\frac{N}{2}\frac{1}{n-2}F_{k}^{k}|(n-1)h_{g}|+\frac{2}{(n-2)}(1+\frac{7}{5}A^{\frac{2}{5}})F_{k}^{k}|u_{\gamma}^{\beta}L_{\beta}^{\gamma}|\no\\
			&+|c_z(1+\frac{7}{5}A^{\frac{2}{5}})\frac{1-t}{n-2}\Delta u e^{-u}P_{i}^{i}|+\frac{(1-t)}{(n-2)}|\sigma_1(W^{\top})\Delta u u_{nn}|\bigg\}-CF_{k}^{k}.\no
		\end{align}
		We point out that the left hand side of (\ref{the chioce of N}) is about $N^3$ and the right hand side is about $N^2$ by virtue of \eqref{eq:tangential control}. 		
		Therefore, collecting all estimates \eqref{eq:matrix non-negative}\eqref{extra boundary term}\eqref{eq:Neumann boundary condition}\eqref{eq:key estimate}  and \eqref{the chioce of N} together,  we obtain 
		\begin{align}
			0\ge & P_{j}^{i}G_{i}^{j}\label{eq:final 1}\\
			\ge & \frac{14}{25}A^{-\frac{3}{5}}\big[u_{nn}\left(f^{2}-\sigma_{2}(W^{\top})-C_{4}\sigma_{1}(W^{\top})\right)-C\sum_{\alpha=1}^{n-1}|W_{\alpha}^{\alpha}|\big]\no\\
			& +\frac{14}{25}A^{-\frac{3}{5}}\bigg(\frac{1-t}{4(n-2)}F_{k}^{k}u_{ni}u_{n}^{i}-\frac{(1-t)}{n-2}\Delta u\sigma_{1}(W^{\top})u_{nn}\bigg)+A\frac{N(n-2)}{2}W_{n}^{n}\no\\
			&-2(1+\frac{7}{5}A^{\frac{2}{5}})\sigma_{1}(W|W_{\alpha}^{\alpha})u_{\alpha}^{\beta}L_{\beta}^{\alpha} -CN\sum_{\alpha}|W^{\alpha}_{\alpha}|-CF_{k}^{k}.\no
		\end{align}
		We remark that  \eqref{eq:final 1} is independent of \eqref{eq:tangential control} when $t=1$.
		By \eqref{upper bound of C3} we have
		
		\begin{equation}
			\frac{1-t}{4(n-2)}F_{k}^{k}u_{ni}u_{n}^{i}-\frac{(1-t)}{n-2}\Delta u\sigma_{1}(W^{\top})u_{nn}>0.
		\end{equation}
		
		Then we know from \eqref{eq:final 1}	 that
		\begin{align}
			0\ge & P_{j}^{i}G_{i}^{j}\no\label{eq:final 2}\\
			\ge &\mathcal{P}+A\frac{N(n-2)}{2}W_{n}^{n},
		\end{align}
		where 
		\begin{align}
			\mathcal{P}:=&\frac{14}{25}A^{-\frac{3}{5}}\big[u_{nn}\left(f^{2}-\sigma_{2}(W^{\top})-C_{4}\sigma_{1}(W^{\top})\right)-C\sum_{\alpha=1}^{n-1}|W_{\alpha}^{\alpha}|\big]\no\\
			& -CN\sum_{\alpha}|W^{\alpha}_{\alpha}|-CF_{k}^{k}.\no
		\end{align}
		%

		\vskip 4pt 
		\emph{Case 1:} $\sigma_{2}(W^{\top})(x_{1})\ge0$. \vskip
		4pt
		
		By definition of $\sigma_2(W)$ and \eqref{u_n alpha} we have
		\[
		0<\sigma_{1}(W^{\top})\leq C\frac{N_{1}^{2}}{W_{n}^{n}}\quad\mathrm{~~and~~}\quad\sigma_{2}(W^{\top})\le\frac{(n-2)\sigma_{1}(W^{\top})^{2}}{2(n-1)}\le C\frac{N_{1}^{4}}{(W_{n}^{n})^{2}}.
		\]

		Actually, now $\lambda(W^{\top})\in\Gamma_{2}^{+}$ implies 
		\begin{equation}
			|W_{\beta}^{\alpha}|\le\sigma_{1}(W^{\top})\le C\frac{N_{1}^{2}}{W_{n}^{n}}.\label{CaseAtangentialbound}
		\end{equation}
		With a choice of large $N$, by \eqref{CaseAtangentialbound} and
		\eqref{lbd:u_nn} we have 
		%
		\begin{equation}
			u_{nn}\big(f^{2}-\sigma_{2}({W}^{\top})-C_{4}\sigma_{1}({W}^{\top})-\frac{C_{4}\sum_{\alpha=1}^{n-1}|W_{\alpha}^{\alpha}|}{u_{nn}}\big)\ge\frac{13}{14}f^{2}W_{nn}
		\end{equation}
		and 
		\begin{equation}
			\mathcal{P}\ge\frac{13}{25}f^{2}W_{nn}A^{-\frac{3}{5}}-C-C\sum_{i}{F}_{i}^{i}.\label{CaseA-P}
		\end{equation}
		We have used the following fact in \eqref{CaseA-P}:	\begin{equation}
			\frac{13}{25}A^{-\frac{3}{5}}W_{nn}f^{2}+\frac{1}{2}NW_{n}^{n}A\ge\frac{1}{2}W_{n}^{n}N^{\frac{3}{8}}f^{\frac{5}{4}}\ge\frac{1}{2}W_{n}^{n}N^{\frac{3}{8}}\inf_{M_{r}}f^{\frac{5}{4}},\label{eq:lower bound of big term}
		\end{equation}
		where the first inequality follows by regarding the left hand side
		as a function of $A$ and seeking its minimum value.
		
		Hence, we combine (\ref{eq:final 1}) (\ref{eq:lower bound of big term})
		and (\ref{CaseA-P}) to show 
		\begin{align*}
			0\ge & P_{j}^{i}G_{i}^{j}\\
			\ge & \frac{1}{2}W_{n}^{n}N^{\frac{3}{8}}\inf_{M_{r}}f^{\frac{5}{4}}-C-C\sum_{i}{F}_{i}^{i}>0,
		\end{align*}
		which yields a contradiction by choosing a larger $N$. We emphasize
		that such an $N$ depends on $\sup_{M_{r}}\frac{1}{f(x,u)}$.
		
		\vskip 4pt \emph{Case 2:} $\sigma_{2}(W^{\top})(x_{1})<0$. \vskip
		4pt
		
		Without loss of generality, we assume that $W_{1}^{1}\ge W_{2}^{2}\ge\cdots\geq W_{n-1}^{n-1}$
		and then we have $W_{1}^{1}>0$ and $W_{n-1}^{n-1}<0$ due to $\sigma_{1}(W^{\top})(x_{1})>0$
		and $\sigma_{2}(W^{\top})(x_{1})<0$.
		
		By $0<\sigma_{1}(W^{\top})=W_{1}^{1}+\cdots+W_{n-1}^{n-1}$, 
		\[
		0<-W_{n-1}^{n-1}\le W_{1}^{1}+\cdots+W_{n-2}^{n-2}\le(n-2)W_{1}^{1}
		\]
		and thus 
		\begin{equation}
			|W_{\alpha}^{\alpha}|\le(n-2)W_{1}^{1}\qquad\mathrm{~~for~~}1\le\alpha\le n-1.\label{the upper bound of W^T by W11}
		\end{equation}
		
		Let $$C^{*}=C_{3}+1.$$
		From (\ref{upper bound of C3}), 
		$$\sigma_{1}(W^{\top})\le C^{*}.$$
		
		
		\emph{Case 2.1:} $W_{1}^{1}>C^{**}:=2C^{\ast}+2\sqrt{(C^{\ast})^{2}+C_{4}C^{\ast}+C_{5}}+5\sqrt{C_{6}}A_{1}^{\frac{3}{10}}+25C_{6}A_{1}^{\frac{3}{5}}$,
		where $C_{5},C_{6}$ are positive constants depending on $g$ and
		$n$.
		
		By 
		\[
		0<W_{1}^{1}+\sigma_{1}(W^{\top}|W_{1}^{1})\le C^{*},
		\]
		we have $-W_{1}^{1}\le\sigma_{1}(W^{\top}|W_{1}^{1})\leq C^{*}-W_{1}^{1}<0$
		by virtue of the choice of $C^{\ast\ast}$.
		
		Combining with the following inequality:
		\[
		-C^{*}W_{n}^{n}\le f^{2}-\sigma_{1}(W^{\top})W_{n}^{n}+\sum_{\alpha}W_{n}^{\alpha}W_{\alpha}^{n}=\sigma_{2}(W^{\top})<0,
		\]
		we obtain 
		\begin{align}
			-C^{*}W_{n}^{n}\le &\sigma_{2}(W^{\top})=W_{1}^{1}\sigma_{1}(W^{\top}|W_{1}^{1})+\sigma_{2}(W^{\top}|W_{1}^{1})\no\\
			\le &W_{1}^{1}(C^{*}-W_{1}^{1})+\frac{n-3}{2(n-2)}\sigma_{1}(W^{\top}|W_{1}^{1})^{2}\no\\
			\le&C^{*}W_{1}^{1}-\frac{n-1}{2(n-2)}(W_{1}^{1})^{2}.\label{polynomialofW11}
		\end{align}
		
		This yields 
		\begin{equation}
			\sqrt{\frac{2(n-2)}{n-1}}\sqrt{C^{*}W_{n}^{n}+\frac{(C^{*})^{2}(n-2)}{2(n-1)}}+\frac{C^{*}(n-2)}{n-1}\ge W_{1}^{1}>0.\label{eq:subcase B.2 the bound of W^11}
		\end{equation}
		By \eqref{the upper bound of W^T by W11} we have 
		\begin{equation}
			|W_{\alpha}^{\alpha}|\le(n-2)\sqrt{\frac{2(n-2)}{n-1}}\sqrt{C^{*}W_{n}^{n}+\frac{(C^{*})^{2}(n-2)}{2(n-1)}}+\frac{C^{*}(n-2)^{2}}{n-1}.\label{tangential W_alphaalpha}
		\end{equation}
		
		By \eqref{polynomialofW11} and \eqref{the upper bound of W^T by W11},
		we have 
		\begin{align}
			& f^{2}-\sigma_{2}({W}^{\top})-C_{4}\sigma_{1}({W}^{\top})-\frac{C_{4}\sum_{\alpha=1}^{n-1}|W_{\alpha}^{\alpha}|}{u_{nn}}\no\label{lower-positive}\\
			\ge & f^{2}+\frac{n-1}{2(n-2)}(W_{1}^{1})^{2}-C^{*}W_{1}^{1}-C_{4}C^{\ast}-C_{5}\no\\
			\ge & f^{2}+\frac{n-1}{4(n-2)}(W_{1}^{1})^{2},
		\end{align}
		where the last inequality follows from the choice of large $C^{\ast\ast}$ such that $W^1_1\ge C^{**}$.
		
		By \eqref{lower-positive} and \eqref{the upper bound of W^T by W11},
		we have 
		\begin{align*}
			\mathcal{P\ge} & \frac{12}{25}A^{-\frac{3}{5}}W_{nn}\left(f^{2}+\frac{n-1}{4(n-2)}(W_{1}^{1})^{2}\right)\\
			& -C_{6}W_{1}^{1}W_{nn}-C_{6}W_{nn}.
		\end{align*}
		
		For sufficiently large $N$ and the choice of $C^{\ast\ast}$ we arrive
		at a contradiction via 
		\begin{align}
			0\ge & P_{j}^{i}G_{i}^{j}\no\label{eq:final subcase B.2}\\
			\ge & \frac{12}{25}A^{-\frac{3}{5}}W_{nn}\left(f^{2}+\frac{n-1}{4(n-2)}(W_{1}^{1})^{2}\right)\nonumber \\
			& -C_{6}W_{1}^{1}W_{nn}-C_{6}W_{nn}+\frac{n-2}{2}NW_{n}^{n}A>0.
		\end{align}
		
		\emph{Case 2.2:} $W_{1}^{1}\le C^{**}=2C^{\ast}+2\sqrt{(C^{\ast})^{2}+C_{4}C^{\ast}+C_{5}}+5\sqrt{C_{6}}A_{1}^{\frac{3}{10}}+25C_{6}A_{1}^{\frac{3}{5}}$.
		
		It follows from \eqref{the upper bound of W^T by W11} that 
		\begin{equation}
			|W_{\alpha}^{\alpha}|\le C\quad\text{for}\quad1\le\alpha\le n-1,\label{upper bound of double in the last case}
		\end{equation}
		and then 
		\[
		\sigma_{1}(W^{\top})=\frac{f^{2}-\sigma_{2}(W^{\top})+\sum_{\alpha}W_{n}^{\alpha}W_{\alpha}^{n}}{W_{n}^{n}}\le\frac{C}{W_{n}^{n}}.
		\]
		From above, choosing $N$ large enough such that 
		\begin{align*}
			& u_{nn}\big(f^{2}-\sigma_{2}({W}^{\top})-C_{4}\sigma_{1}({W}^{\top})-\frac{C_{4}\sum_{\alpha=1}^{n-1}|W_{\alpha}^{\alpha}|}{u_{nn}}\big)\\
			> & u_{nn}\big(f^{2}-\frac{C_{7}}{W_{n}^{n}}\big)>\frac{13}{14}W_{n}^{n}\inf f^{2}
		\end{align*}
		and then 
		\begin{equation}
			\mathcal{P}\ge\frac{13}{25}f^{2}W_{nn}A^{-\frac{3}{5}}-C-C\sum_{i}{F}_{i}^{i}.\label{CaseB-P lastsubcase}
		\end{equation}
		By (\ref{eq:final 2})(\ref{CaseB-P lastsubcase})(\ref{eq:lower bound of big term})
		we know that 
		\begin{align*}
			0 & \ge P_{j}^{i}G_{i}^{j}\\
			&\ge\frac{13}{25}W_{n}^{n}A^{-3/5}\inf_{M_{r}}f^{2}+\frac{(n-2)}{2}NW_{n}^{n}A-C\sum_{i}{F}_{i}^{i}\\
			&\ge\frac{1}{2}W_{n}^{n}N^{\frac{3}{8}}\inf_{M_{r}}f^{\frac{5}{4}}-C\sum_{i}{F}_{i}^{i}>0.
		\end{align*}
		This is a contradiction.
		Therefore, the maximum point can only happen at the boundary
		and thus at $x_{0}$. 
		
\medskip
		In the following,  we give the proof of  {\bf Claim 1} and  {\bf Claim 2}.	
\medskip		
		For brevity, we let $\widetilde{\varphi}_{i}=\varphi_{i}-u^{k}(x_{n})_{ki}$
		and write 
		\begin{align*}
			J:=& P_{j}^{i}(u_{ni}+u^{k}(x_{n})_{ki}-\varphi_{i})(u_{n}^{j}++u^{l}(x_{n})_{l}^{j}-\varphi^{j})\\
			= & P_{j}^{i}\widetilde{\varphi}_{i}\widetilde{\varphi}^{j}-2P_{j}^{i}u_{ni}\widetilde{\varphi}^{j}+P_{j}^{i}u_{ni}u_{n}^{j}\\
			=:&J_1+J_2+J_3.
		\end{align*}
		Notice that 
		\begin{align}\label{eq:varphi}
			J_1:=&P_{j}^{i}\widetilde{\varphi}_{i}\widetilde{\varphi}^{j}\no\\
			\ge & \big(\sigma_{1}(W^{\top})+\frac{1-t}{n-2}\sum_{k}F_{k}^{k}\big)\widetilde{\varphi}_{n}^{2}+\big[u_{nn}+(1-t)\frac{(n-1)\Delta u}{n-2}+\frac{1-t}{n-2}\sum_{k}F_{k}^{k}\big]\widetilde{\varphi}_{\alpha}\widetilde{\varphi}^{\alpha}\no\\
			&-C\sum_{\alpha}|u_{\alpha}^{\alpha}|-\sum_{\alpha}|u_{n\alpha}|.
		\end{align}
		Using $F_{\alpha}^{\beta}u_{n\beta}=\sigma_{1}(W)u_{n\alpha}-W_{\alpha}^{\beta}u_{n\beta}=(1+\frac{n(1-t)}{n-2})u_{nn}+C\sum_{\alpha}|W^{\alpha}_{\alpha}|+C\sum_{\alpha}|u^{\alpha}_{\alpha}|$, we obtain
		\begin{align*}
			& 2{W}_{\alpha}^{n}u_{nn}\widetilde{\varphi}^{\alpha}-2\bigg(F_{\alpha}^{\beta}+\frac{1-t}{n-2}\sum_{k}F_{k}^{k}g_{\alpha}^{\beta}\bigg)u_{n\beta}\widetilde{\varphi}^{\alpha}\\
			= & 2\widetilde{\varphi}^{\alpha}\bigg(u_{nn}(u_{n\alpha}+u_{n}u_{\alpha}+A_{n\alpha}^{t})-(F_{\alpha}^{\beta}+\frac{1-t}{n-2}\sum_{k}F_{k}^{k}g_{\alpha}^{\beta})u_{n\beta}\bigg)\\
			\ge & 2\widetilde{\varphi}^{\alpha}u_{nn}(u_{n}u_{\alpha}+A_{n\alpha}^{t})-\frac{2n(1-t)}{n-2}\widetilde{\varphi}^{\alpha}u_{nn}u_{n\alpha}-\frac{2(1-t)}{n-2}\sum_{k}F_{k}^{k}u_{n\alpha}\widetilde{\varphi}^{\alpha}-C\sum_{\alpha}|W_{\alpha}^{\alpha}|.
		\end{align*}
		Thus, 
		\begin{align}\label{eq:mixed term}
			&J_2:=-2P_{j}^{i}u_{ni}\widetilde{\varphi}^{j}\\
			= & -2\big(\sigma_{1}({W}^{\top})+\frac{1-t}{n-2}\sum_{k}F_{k}^{k}\big)u_{nn}\widetilde{\varphi}^{n}+2{W}_{n}^{\alpha}u_{n\alpha}\widetilde{\varphi}^{n}+2{W}_{\alpha}^{n}u_{nn}\widetilde{\varphi}^{\alpha}-2\bigg(F_{\alpha}^{\beta}+\frac{1-t}{n-2}\sum_{k}F_{k}^{k}g_{\alpha}^{\beta}\bigg)u_{n\beta}\widetilde{\varphi}^{\alpha} \nonumber\\
			\geq & -2\big(\sigma_{1}({W}^{\top})+\frac{1-t}{n-2}\sum_{k}F_{k}^{k}\big)u_{nn}\widetilde{\varphi}^{n}+2{W}_{n}^{\alpha}u_{n\alpha}\widetilde{\varphi}^{n}+2\widetilde{\varphi}^{\alpha}u_{nn}(u_{n}u_{\alpha}+A_{n\alpha}^{t})\nonumber \\
			& -\frac{2n(1-t)}{n-2}\widetilde{\varphi}^{\alpha}u_{nn}u_{n\alpha}-\frac{2(1-t)}{n-2}\sum_{k}F_{k}^{k}u_{n\alpha}\widetilde{\varphi}^{\alpha}-C\sum_{\alpha}|W_{\alpha}^{\alpha}|\nonumber \\
			\ge & -2\big(\sigma_{1}({W}^{\top})+\frac{1-t}{n-2}\sum_{k}F_{k}^{k}\big)u_{nn}\widetilde{\varphi}^{n}-u_{nn}\sum_{\alpha}\widetilde{\varphi}_{\alpha}\widetilde{\varphi}^{\alpha}-\sum_{\alpha}(u_{n}u_{\alpha}+A_{n\alpha}^t)(u_{n}u^{\alpha}+A_{n}^{t, \alpha})u_{nn}\nonumber \\
			&-C\sum_{\alpha}|W_{\alpha}^{\alpha}|-C,\nonumber
		\end{align}
		where the last inequality holds due to the Cauchy inequality. 
		
		We turn to deal with the most difficult term: 
		
		\begin{align}
			& J_3:=P_{j}^{i}u_{ni}u_{n}^{j}\label{eq:quadratic 2} \\
			= & (F_{j}^{i}+\frac{1-t}{n-2}F_{k}^{k}g_{j}^{i})u_{ni}u_{n}^{j}\no\\
			= & \frac{1-t}{n-2}F_{k}^{k}u_{ni}u_{n}^{i}+F_{n}^{n}u_{nn}u_{nn}+F_{\beta}^{\alpha}u_{n\alpha}u_{n}^{\beta}+2F_{\beta}^{n}u_{nn}u_{n}^{\beta}\no\\
			= & \frac{1-t}{n-2}F_{k}^{k}u_{ni}u_{n}^{i}+\sigma_{1}(W^{\top})u_{nn}^{2}+(\mathrm{tr}_{g}W\delta_{\beta}^{\alpha}-W_{\beta}^{\alpha})u_{n\alpha}u_{n}^{\beta}+2(-W_{n\beta})u_{nn}u_{n}^{\beta}\no\\
			\ge & \frac{1-t}{n-2}F_{k}^{k}u_{ni}u_{n}^{i}+u_{nn}\bigg(\sigma_{1}(W^{\top})u_{nn}+\left(1+(1-t)\frac{n-1}{n-2}\right)u_{n\alpha}u_{n}^{\alpha}-2W_{\alpha n}u_{n}^{\alpha}\bigg)\nonumber \\
			& -C\sum_{\gamma}|u^{\gamma}_{\gamma}|.\no
		\end{align}
		Thus, by (\ref{eq:varphi})(\ref{eq:mixed term})(\ref{eq:quadratic 2}), for large $N$ we have 
		\begin{align}	
			J\ge & \frac{1-t}{n-2}F_{k}^{k}u_{ni}u_{n}^{i}+u_{nn}\bigg(\sigma_{1}(W^{\top})u_{nn}+(1+(1-t)\frac{n-1}{n-2})u_{n\alpha}u_{n}^{\alpha}-2W_{\alpha n}u_{n}^{\alpha}\bigg)\nonumber \\
			& -2\big(\sigma_{1}({W}^{\top})+\frac{1-t}{n-2}\sum_{k}F_{k}^{k}\big)u_{nn}\widetilde{\varphi}^{n}-\sum_{\alpha}(u_{n}u_{\alpha}+A^t_{n\alpha})(u_{n}u^{\alpha}+A_{n}^{t,\alpha})u_{nn}\nonumber \\
			& +(1-t)\big[\frac{(n-1)\Delta u}{n-2}+\frac{1}{n-2}\sum_{k}F_{k}^{k}\big]\widetilde{\varphi}_{\alpha}\widetilde{\varphi}^{\alpha}\label{eq:for sigma_1(W^T) small use}\\
			&-C\sum_{\alpha}|u_{\alpha}^{\alpha}|-C\sum_{\alpha}|W_{\alpha}^{\alpha}|-C(1-t)N\nonumber \\
			\ge & \frac{1-t}{n-2}F_{k}^{k}u_{ni}u_{n}^{i}+u_{nn}\sigma_{1}(W^{\top})u_{nn}-CW_{nn}^{2}.\label{eq:for large sigma_1(W^T) large use}
		\end{align}
		where the last inequality is due to \eqref{eq:doubel tangential initial bound}. Now Claim (1) has been proved. However, \eqref{eq:for large sigma_1(W^T) large use} is not enough for later use, we need the following refined estimates:
		\begin{align}\label{eq:main-1}
			&\sigma_{1}(W^{\top})u_{nn}+u_{n\alpha}u_{n}^{\alpha}+(1-t)u_{n\alpha}u_{n}^{\alpha}\frac{n-1}{n-2}-2W_{\alpha n}u_{n}^{\alpha}\no\\
			=&	f-\sigma_{2}(W^{\top})+\bigg(u_{n}u_{\alpha}+A_{n\alpha}^{t}\bigg)^{2}\bigg(1+\frac{(1-t)(n-1)}{n-2}\bigg)\no\\
			&+\frac{(1-t)(n-1)}{n-2}W_{n\alpha}^{2}
			-2W_{n\alpha}(u_{n}u_{\alpha}+A_{n\alpha}^{t})\frac{(1-t)(n-1)}{n-2}-\frac{(1-t)}{n-2}\Delta u\sigma_{1}(W^{\top})\no\\
			&+(-u_{n}^{2}+\frac{2-t}{2}|\nabla u|^{2}-A_{nn}^{t})\sigma_{1}(W^{\top})\no\\
			\ge &	f-\sigma_{2}(W^{\top})+\bigg(u_{n}u_{\alpha}+A_{n\alpha}^{t}\bigg)\bigg(u_{n}u^{\alpha}+A_{n}^{t,\alpha}\bigg)-\frac{(1-t)}{n-2}\Delta u\sigma_{1}(W^{\top})\no\\
			&+(-u_{n}^{2}+\frac{2-t}{2}|\nabla u|^{2}-A_{nn}^{t})\sigma_{1}(W^{\top}).
		\end{align}

		Therefore, by (\ref{eq:for sigma_1(W^T) small use})(\ref{eq:main-1}),
		we have 
		\begin{align}\label{eq:key estimate-1}
			& P_{j}^{i}(u_{ni}+u^{k}(x_{n})_{ki}-\varphi_{i})(u_{n}^{j}+u^{l}(x_{n})_{l}^{j}-\varphi^{j})\nonumber\\
			\ge & u_{nn}\bigg(f-\sigma_{2}(W^{\top})+(-u_{n}^{2}+\frac{2-t}{2}|\nabla u|^{2}-A_{nn}^{t})\sigma_{1}(W^{\top})-2\sigma_{1}(W^{\top})\varphi_{n}\bigg)\no\\
			& +\frac{1-t}{(n-2)}F_{k}^{k}u_{ni}u_{n}^{i}-\frac{(1-t)}{n-2}\Delta u\sigma_{1}(W^{\top})u_{nn}\no\\
			& -C(1-t)F_{k}^{k}-C\sum_{\alpha}|W_{\alpha}^{\alpha}|-2\frac{1-t}{n-2}\sum_{k}F_{k}^{k}u_{nn}\widetilde{\varphi}^{n}-C.
		\end{align}
		Then {\bf Claim 1} and {\bf Claim 2}  are proved.
	\end{proof}
	
	Now the proof of Theorem \ref{thm:global C2 estimate} is completed by Theorem \ref{Thm:reduction} and Theorem \ref{high dimen-double normal}. 
	\begin{remark}
		For $t < 1$ and any prescribed mean curvature $c$, the $C^2$ estimates hold independently of $\inf_M f$. This independence is a result of the fact that $\frac{1-t}{n-2}F_{k}^{k}u_{ni}u_{n}^{i}$ dominates in the expressions. The proof of double normal derivative reveals that its growth is of order $N^3$, while the other terms are of order $N^2$.
	\end{remark}
	
	\subsection{$\sigma_2$-nonlinear eigenvalue problem}\label{Subsect2.3}
	In this subsection we demonstrate that there exists a conformal  metric  having a positive $\sigma_2$-curvature and zero boundary mean curvature.

	Now we can prove Theorem \ref{theorem:sigma_2_1st_bdry_eigenvalue-1}.	
	\begin{proof}[Proof of Theorem \ref{theorem:sigma_2_1st_bdry_eigenvalue-1}: ]
		
		\emph{Step 1:} We prove that
		for every $\varepsilon>0$, there exists a solution
		$u$ satisfying
		\[
		\begin{cases}
			\sigma_{2}(A^t_{g_{u}})=e^{\varepsilon u} & \quad\mathrm{~~in~~}M,\\
			\frac{\partial u}{\partial\vec{n}}=-h_{g} & \quad\mathrm{~~on~~}\partial M.
		\end{cases}
		\]
		
		To get the existence, we consider the following equation 
		\begin{equation}\label{eq:continuity of equation}
			\begin{cases}
				\sigma_{2}(A^t_{g_{u}})=\left(s+(1-s)f(x)\right)e^{\varepsilon u}  & \qquad\mathrm{~~in~~}M,\\
				\frac{\partial u}{\partial\vec{n}}=-sh_{g} & \qquad\mathrm{~~on~~}\partial M,
			\end{cases}
		\end{equation}
		where $s\in[0,1]$ and $f(x)=\sigma_{2}(g^{-1}A_{g}^{t}).$
		
		Denote 
		\[
		\mathbb{I}=\{s\in[0,1];\mathrm{~~there~exists~a~solution~}u\mathrm{~satisfying~(\ref{eq:continuity of equation})~with~}\lambda(g^{-1}A^t_{g_{u}})\in\Gamma_{2}^{+}\}.
		\]
		
		At $s=0$, $u=0$ is the unique solution and the linearized operator
		is elliptic and invertible. It suffices
		to show the closeness by a priori estimates, thus the proof is completed
		by the continuity method.
		$\vspace{4pt}$	
		
		Choose a smooth function $l$ such that $\int_{M}l\ud\mu_{g}=\int_{\partial M}(h_{g}+1)\ud\sigma_{g}$
		and find a unique (up to a constant) smooth solution $v$ to the following
		PDE with Neumann boundary condition 
		\[
		\begin{cases}
			\Delta v=sl & \qquad\mathrm{~~in~~}M,\\
			\frac{\partial v}{\partial\vec{n}}=-sh_{g}-s & \qquad\mathrm{~~on~~}\partial M.
		\end{cases}
		\]
		
		We first observe that $\min_{\overline{M}}(v-u)$ can not happen on
		boundary $\partial M$ due to $\frac{\partial(v-u)}{\partial\vec{n}}=-s$
		on $\partial M$ for $t>0.$ Thus, we may assume $\min_{\overline{M}}(v-u)=(v-u)(x_{0})$
		for some $x_{0}\in M$ and $\nabla^{2}(v-u)(x_{0})\ge0$ and $\nabla v(x_{0})=\nabla u(x_{0}).$
		Then at $x_{0}$, for $1-\frac{n}{2}-\frac{1-t}{2}n<0,$ we have 
		\begin{align*}
			&\big((1+\frac{n(1-t)}{n-2})\Delta v+\sigma_{1}(A_{g})\big)^{2}\frac{C_{n}^{2}}{n^{2}}\\
			&=\sigma_{1}^{2}(A_{g_v}^{t})\frac{C_{n}^{2}}{n^{2}}\ge\sigma_{2}(A_{g_v})\ge\sigma_{2}(A_{g_u})\\
			&=\left(s+(1-s)f(x)\right)e^{\varepsilon u} 
		\end{align*}
		whence, 
		\[
		e^{\varepsilon u(x_{0})}\le C,
		\]
		where $C$ is a constant independent of $s$ and $\varepsilon$.
		
		Notice that 
		\[
		(v-u)(x)\ge(v-u)(x_{0})
		\]
		implies 
		\[
		u(x)\le u(x_{0})+v(x)-v(x_{0}).
		\]
		
		If $h_{g}\ge0$, then we let $\min_{\overline{M}}u=u(x_{1})$ for
		some $x_{1}\in M$ and for $t>0$, $\nabla u(x_{1})=0$ and $\nabla^{2}u(x_{1})\ge0.$
		Then at $x_{1}$, we have
		
		\[
		\left(s+(1-s)f(x)\right)e^{\varepsilon u}=\sigma_{2}(A_{g_u}^{t})\ge\sigma_{2}(A_{g}^{t}),
		\]
		whence, 
		\[
		e^{\ve u(x_{1})}\ge \min_{\overline M} \frac{\sigma_{2}(A_{g}^{t})}{s+(1-s)f(x)}.
		\]
		
		Consequently, $u$ is uniformly bounded in $s\in[0,1]$. 
		Thus by the $C^1$ estimates  in \cite{Jin-Li-Li}, Theorem \ref{Thm:reduction} and Theorem \ref{high dimen-double normal}, we have 
		\begin{align}
			& \|\nabla u\|_{C^{0}(\overline{M})}\le C_{1},
		\end{align}
		and 
		\begin{align}
			& \|\nabla^{2}u\|_{C^{0}(\overline{M})}\le C_2,
		\end{align}
		where $C_{1}$ depends on $|f|_{C^1(M)}$ and $\sup_{M}e^{\varepsilon u}$, $C_{2}$ depends on $\frac{1}{\inf_{M}(s+(1-s)f)e^{\varepsilon u}}$, $\sup_{M}e^{\varepsilon u}$ and
		$|f|_{C^2(M)}$.
		
		Therefore, we have 
		\[
		\|\nabla u\|_{C^{0}(\overline{M})}+\|\nabla^{2}u\|_{C^{0}(\overline{M})}\le C,
		\]
		where $C$ is independent of $s$ and $\ve$.
		
		By the continuity method, we know that for every $\varepsilon>0$,
		there exists a solution $u$ to 
		\[
		\begin{cases}
			\sigma_{2}(g^{-1}A^t_{g_{u}})=e^{\varepsilon u} & \qquad\mathrm{~~in~~}M,\\
			\frac{\partial u}{\partial\vec{n}}+h_g=0 & \qquad\mathrm{~~on~~}\partial M,
		\end{cases}
		\]
		satisfying 
		\[
		\|\nabla u\|_{C^{0}(\overline{M})}+\|\nabla^{2}u\|_{C^{0}(\overline{M})}\le C,
		\]
		where $C$ is independent of $\varepsilon.$
		
		\vskip 3pt
		\emph{Step 2:} 		
		Rewrite the above equation as 
		\[
		\begin{cases}
			\sigma_{2}(A^t_{g_{u}})=e^{\varepsilon(u-\bar{u})+\varepsilon\bar{u}} & \qquad\mathrm{~~in~~}M,\\
			\frac{\partial u}{\partial\vec{n}}+h_g=0 & \qquad\mathrm{~~on~~}\partial M,
		\end{cases}
		\]
		where $\bar{u}=\fint_{M}u\ud\mu_{g}$. Then, for any sequence $\{\ve_{i}\}$
		with $\ve_{i}\to0$, there holds $\|\varepsilon_{i}\nabla u\|_{C^{0}(\overline{M})}\rightarrow0$
		and thereby up to a subsequence, $\ve_{i}u=\varepsilon_{i}(u-\bar{u})+\varepsilon_{i}\bar{u}\rightarrow\lambda\in\R.$
		
		Let $v_{\varepsilon_{i}}=u-\bar{u}$ satisfy 
		\[
		\begin{cases}
			\sigma_{2}(A^t_{g_{v_{\varepsilon}}})=e^{\varepsilon_{i}v_{\varepsilon_{i}}+\varepsilon_{i}\bar{u}} & \qquad\mathrm{~~in~~}M,\\
			\frac{\partial v_{\varepsilon_{i}}}{\partial\vec{n}}+h_g=0 & \qquad\mathrm{~~on~~}\partial M,
		\end{cases}
		\]
		by virtue of the fact that $A^t_{g_u}=A^t_{g_{v_{\ve_{i}}+\bar{u}}}=A^t_{g_{v_{\ve_{i}}}}$,
		and $\|v_{\varepsilon_{i}}\|_{C^{2}(\overline{M})}\le C$. Thus for
		any $\alpha\in(0,1)$, there holds 
		\[
		\|v_{\varepsilon_{i}}\|_{C^{2,\alpha}(\overline{M})}\le C
		\]
		and for $\gamma\in(0,\alpha)$, as $\ve_{i}\to0$ we have 
		\[
		v_{\varepsilon_{i}}\rightarrow v\qquad\mathrm{~~in~~}C^{2,\gamma}(\overline{M}),
		\]
		where $v$ satisfies 
		\[
		\begin{cases}
			\sigma_{2}(A_{g_{v}}^{t})=e^{\lambda} & \qquad\mathrm{~~in~~}M,\\
			\frac{\partial v}{\partial\vec{n}}+h_g=0 & \qquad\mathrm{~~on~~}\partial M.
		\end{cases}
		\]
		
		Moreover, the uniqueness of the constant $e^{\lambda}$ follows from
		the strong maximum principle and Hopf lemma for the elliptic equation. 
	\end{proof}

	
	
	Following the same lines of the proof of \cite[Theorem 1.4]{Jin-Li-Li},  we know that 
the gradient estimates hold depending on $\sup_M e^{\varepsilon u}$. 	
	
	\section{$\sigma_2$-Yamabe Problem with prescribed mean curvature}\label{Sect3}
	
	As the  $\sigma_2$ curvature equation on three-manifolds is of special interest, in this section we focus on the case $t=1$:  For a positive
	function $f\in C^{2}(\overline{M}),$
	\begin{equation}
		\begin{cases}
			\sigma_{2}^{1/2}(g^{-1}A_{g_u})=f(x,u) &\qquad\mathrm{~~in~~}M, \quad \lambda(g^{-1}A_{g_u})\in \Gamma_2^+,\\
			\frac{\partial u}{\partial \vec{n}}+h_{g}=c(x)e^{-u} & \qquad\mathrm{~~on~~}\partial M.
		\end{cases}\label{eq:main equation 2}
	\end{equation}
	
Surprisingly, on three-manifolds, the double normal derivative can be obtained without estimates of second tangential derivatives. This peculiar phenomenon marks a departure from previous discussions.

We will begin by deriving the local double normal derivative  on $\partial M$. Subsequently, under certain natural conformally invariant conditions, we establish the existence of a solution to (\ref{eq:main equation 2}) with $f(x,u)=f(x)e^{-2u}$ in $M$, $h_{g}=0$ and $c=0$ on $\partial M$
.
	\subsection{Local double normal derivative on boundary}
	\begin{thm}\label{three-double normal}
		On $(M^{3},g)$, suppose $c(x)\ge0$ on $\partial M$. Let $u$ be
		a $C^{2}$ solution to equation (\ref{eq:main equation 2}). For
		any $\mathcal{O}\subset\mathcal{O}_{1}\subset\overline{M}$, there
		holds $$u_{nn}\le C,\quad \mathrm{on}\quad  \mathcal{O}\cap\partial M,$$ where $C$ is
		a positive constant depending on $g,f,\sup_{\mathcal{O}_{1}\cap\partial M}(|\tilde{\nabla}^{2}c|+|\tilde{\nabla}c|+c)e^{-u}$,
		$\sup_{\mathcal{O}_{1}}|\nabla u|$, $\sup_{\mathcal{O}_{1}}(f(x,u)+|f_{x}|+|f_{z}|+|f_{xx}|+|f_{xz}|+|f_{xz}|)$
		and $\sup_{\mathcal{O}_{1}}\frac{1}{f}.$
	\end{thm}
	
	The proof for Theorem \ref{three-double normal} is almost same as before.  We just sketch the proof and point out the difference. 
	\begin{proof}
		Let $n=3$ and fix $x_{0}\in\partial M$. In a tubular neighborhood
		near $\partial M$ we define 
		\begin{align*}
			G(x)= & \langle\nabla u,\nabla d_{g}(x,\partial M)\rangle-ce^{-u}+h_{g}+\left(\langle\nabla u,\nabla d_{g}(x,\partial M)\rangle-ce^{-u}+h_{g}\right)^{7/5}\\
			& -\frac{1}{2}Nd_{g}(x,\partial M)-N_{1}d_{g}^{2}(x_{0},x),
		\end{align*}
		where $N,N_{1}\in\mathbb{R}_{+}$ are to be determined later and $N>N_{1}$.
		
		It suffices to consider $G$ in $M_{r}$ for sufficiently small $r$
		together with the choice of $N_{1}$ such that 
		\begin{align}
			N_{1}r^{2}> & \sup_{M_{r}}|\langle\nabla u,\nabla d_{g}(x,\partial M)\rangle-ce^{-u}+h_{g}+(\langle\nabla u,\nabla d_{g}(x,\partial M)\rangle-ce^{-u}+h_{g})^{7/5}|\no\\
			:= & C_{0}(r).\label{eq:condtion 0-1-1}
		\end{align}
		
		Due to the same reason as before, our purpose is to show that $G$
		achieves its maximum value at $x_{0}$ for some appropriate $N$ and
		$N_{1}$.
		
		By contradiction, we assume 
		\[
		\max_{\overline{M_{r}}}G=G(x_{1})>0\quad\mathrm{~~for~~}\quad x_{1}\in\mathring{M}_{r}.
		\]

		For $n=3,\sigma_{1}(W^{\top})=W_{1}^{1}+W_{2}^{2}>0$. Without loss
		of generality, we assume that $W_{1}^{1}\ge W_{2}^{2}$ and $W_{1}^{1}>0$
		at $x_{1}$. Actually, there hold $W_{1}^{1}\ge|W_{2}^{2}|$ and $F_{1}^{1}\le F_{2}^{2}$
		at $x_{1}$. Roughly, we know that $\frac{1}{4}N\le u_{nn}\le N$ and $\frac{1}{2}W_{nn}\le u_{nn}\le 2W_{nn}$.
		
		By the similar argument of the proof of Theorem \ref{high dimen-double normal}, similar by \eqref{eq:final 1},  we have 
		\begin{align}\label{eq:final 1-1}
			0\ge & \frac{14}{25}A^{-\frac{3}{5}}\big[u_{nn}\left(f^{2}-\sigma_{2}(W^{\top})-C_{4}\sigma_{1}(W^{\top})\right)-CW_{1}^{1}\big]\no\\
			&+A\frac{N(n-2)}{2}W_{n}^{n}
			-CNW^{1}_{1}-CF_{k}^{k}.\no\\
			&=:\mathbb{P}+A\frac{N(n-2)}{2}W_{n}^{n}, 
		\end{align}	
where we have used that 
\begin{align*}
&2(1+\frac{7}{5}A^{\frac{2}{5}})\sigma_{1}(W|W_{\alpha}^{\alpha})u_{\alpha}^{\beta}L_{\beta}^{\alpha}\\
\le& C(W^1_1+W^3_3)|W^2_2|+C(W^2_2+W^3_3)|W^1_1|\\
\le & CW^3_3W^1_1+CW^1_1|W^2_2|\\
\le & CNW^1_1+C\big|f^2-\sigma_{1}(W^{\top})W_{n}^{n}+\sum_{\alpha}W_{n}^{\alpha}W_{\alpha}^{n}\big|\\
\le & CNW^1_1+C.
\end{align*}

		Consider $\frac{13}{25}A^{-\frac{3}{5}}W_{nn}f^{2}+\frac{1}{2}NW_{n}^{n}A$
		as a function of $A$, we know that
		
		\begin{equation}
			\frac{13}{25}A^{-\frac{3}{5}}W_{nn}f^{2}+\frac{1}{2}NW_{n}^{n}A\ge(\frac{1}{2})^{\frac{13}{8}}W_{n}^{n}N^{\frac{3}{8}}f^{\frac{5}{4}}\ge(\frac{1}{2})^{\frac{13}{8}}W_{n}^{n}N^{\frac{3}{8}}\inf f^{\frac{5}{4}}.\label{eq:lower bound of big term-1}
		\end{equation}
		This fact will be used later. 
		\vskip 4pt \emph{Case 1: } $W^2_2(x_1)\geq0$.
		\vskip 4pt
		
		By definition of $\sigma_2(W)$ and the first derivatives of $G$, we have 
		
		\[
		0\leq W_{1}^{1}W_{2}^{2}=\sigma_{2}(W^{\top})\leq CN_{1}^{2}
		\]
		and then 
		\begin{equation}
			0\leq W_{2}^{2}\le W_{1}^{1}\le\sigma_{1}(W^{\top})\leq C\frac{N_{1}^{2}}{W_{n}^{n}}\quad\mathrm{~~and~~}\quad\sigma_{2}(W^{\top})\leq C\frac{N_{1}^{4}}{(W_{n}^{n})^{2}}.\label{CaseAtangentialbound-1}
		\end{equation}
		With a  choice of large $N$, by \eqref{CaseAtangentialbound-1} we have 
		\begin{equation}
			u_{nn}\big(f^{2}-\sigma_{2}({W}^{\top})-C_{4}\sigma_{1}({W}^{\top})-\frac{C_{4}W^1_1}{u_{nn}}\big)\ge\frac{13}{14}f^{2}W_{nn}.
		\end{equation}
		and then by (\ref{eq:final 1-1})(\ref{CaseAtangentialbound-1})(\ref{eq:lower bound of big term-1}), 
		we obtain 
		\begin{align*}
			0\ge & F_{j}^{i}G_{i}^{j}\\
			\ge & (\frac{1}{2})^{\frac{13}{8}}W_{n}^{n}N^{\frac{3}{8}}\inf f^{\frac{5}{4}}-C-CN>0,
		\end{align*}
		which yields a contradiction by a further large $N$ also depending
		on $\sup_{M_{r}}\frac{1}{f(x,u)}$.
		
		\vskip 4pt \emph{Case 2:} $W^2_2(x_{1})<0$. \vskip
		4pt
		
		Notice that $\frac{\partial\sigma_{2}(W)}{\partial W_{1}^{1}}>0$
		implies $0>W_{2}^{2}>-W_{n}^{n}$ at $x_{1}$. Moreover, 
		\begin{equation}
			|W_{2}^{2}|\le W_{1}^{1}\le\sigma_{1}(W)=\sigma_{1}(W^{\top})+W_{nn}.\label{upperboundof W}
		\end{equation}
		
		By the definition of $\sigma_2$, we write 
		\begin{equation}
			\sigma_{1}(W^{\top})=\frac{(W_{2}^{2})^{2}+(\sum_{\alpha}W_{n}^{\alpha}W_{\alpha}^{n}+f^{2})}{W_{2}^{2}+W_{n}^{n}}.\label{tangent_trace_W}
		\end{equation}
		
		\emph{Case 2.1:} $\sigma_{1}(W^{\top})\ge C^{*}=10C_{2}A_{1}^{\frac{3}{5}}+1$,
		where $C_{2}$ is a positive constant depending on the second fundament
		of the boundary and $|u|_{C^{1}}$.
		
		By \eqref{upperboundof W} we choose a sufficiently large $N$ such
		that 
		\begin{equation}
			\sigma_{1}(W^{\top})W_{n}^{n}-\sum_{\alpha}W_{n}^{\alpha}W_{\alpha}^{n}-C_{4}\sigma_{1}({W}^{\top})-C_{4}\frac{W_{1}^{1}}{u_{nn}}\ge\frac{3}{4}\sigma_{1}(W^{\top})\frac{(W_{n}^{n})^{2}}{u_{nn}}.\label{lower bound of main term 1}
		\end{equation}
		As 
		\[
		f^{2}-\sigma_{2}(W^{\top})=\sigma_{1}(W^{\top})W_{n}^{n}-\sum_{\alpha}W_{n}^{\alpha}W_{\alpha}^{n},
		\]
		and \eqref{lower bound of main term 1}, we have 
		\begin{align}
			\mathbb{P}\ge & \frac{13}{25}A^{-\frac{3}{5}}u_{nn}\left(\sigma_{1}(W^{\top})W_{n}^{n}-\sum_{\alpha}W_{n}^{\alpha}W_{\alpha}^{n}-C_{4}\sigma_{1}({W}^{\top})-C_{4}\frac{W_{1}^{1}}{u_{nn}}\right)\nonumber \\
			& -CW_{n}^{n}W_{1}^{1}-C(W_{1}^{1}+W_{n}^{n})\nonumber \\
			\ge & \frac{42}{100}A^{-\frac{3}{5}}W_{n}^{n}W_{n}^{n}\sigma_{1}({W}^{\top})-C_{2}W_{n}^{n}W_{1}^{1}-C_{2}(\sigma_{1}({W}^{\top})+W_{n}^{n})\nonumber \\
			\ge & \frac{21}{100}A^{-\frac{3}{5}}W_{n}^{n}W_{n}^{n}\sigma_{1}({W}^{\top}),\label{the lower bound of P-CaseB.1-3}
		\end{align}
		where we taking $N$ sufficiently large such that 
		\[
		\frac{21}{100}A^{-\frac{3}{5}}W_{n}^{n}W_{n}^{n}\sigma_{1}({W}^{\top})\ge C_{2}W_{n}^{n}W_{1}^{1}+C_{2}(\sigma_{1}({W}^{\top})+W_{n}^{n}).
		\]
		Here we have used $\sigma_{1}(W^{\top})>C^{*}$ in the above inequality.
		
		%
		Then it follows from (\ref{eq:final 1-1})(\ref{the lower bound of P-CaseB.1-3}) that
		\[
		0\ge F_{j}^{i}G_{i}^{j}>\frac{1}{2}NW_{n}^{n}A>0.
		\]
		This yields a contradiction.

		\emph{Case 2.2:} $\sigma_{1}(W^{\top})<C^{*}.$
		
		Notice that 
		\[
		0<W_{1}^{1}+W_{2}^{2}\le C^{*},\quad\sum_{i}F_{i}^{i}\le2W_{n}^{n}+2C^{*}.
		\]
		Thus, for a large $N$ we have 
		\begin{equation}
			0<\sum F_{i}^{i}\le4W_{n}^{n}.\label{upper bound of F^ii}
		\end{equation}
		It follows from \eqref{tangent_trace_W} that 
		\[
		\frac{C^{*}}{2}-\sqrt{\frac{(C^{\ast})^{2}}{4}+C^{*}W_{n}^{n}}\le W_{2}^{2}<0
		\]
		and then in Case 2.2, we have
		\begin{equation}
			0<W_{1}^{1}\le\sqrt{\frac{(C^{\ast})^{2}}{4}+C^{*}W_{n}^{n}}+\frac{C^{*}}{2}.\label{eq:subcase B.2 the bound of W^11-1}
		\end{equation}

		By the choice of large $N$ and $|W_{2}^{2}|\le W_{1}^{1}$, we have
		\begin{align*}
			\mathbb{P}\ge & \frac{13}{25}A^{-\frac{3}{5}}W_{nn}\left(f^{2}-\sigma_{2}({W}^{\top})-CW_{1}^{1}+CW_{2}^{2}\right)-CW_{n}^{n}W_{1}^{1}-C(\sum_{i}{F}_{i}^{i})\\
			\ge & \frac{13}{25}A^{-\frac{3}{5}}W_{nn}\left(f^{2}-\sigma_{2}({W}^{\top})-C_{5}W_{1}^{1}+C_{5}W_{2}^{2}\right)-C(\sum_{i}{F}_{i}^{i})
		\end{align*}
		where $C_{5}$ depends on $|g|_{C^{2}}$ and $|u|_{C^{1}}.$
		
		Now we know from \eqref{eq:final 1-1} that
		\begin{align}
			0\ge & F_{j}^{i}G_{i}^{j}\no\label{eq:final subcase B.2-1}\\
			\ge & \frac{13}{25}A^{-\frac{3}{5}}W_{nn}\left(f^{2}-\sigma_{2}({W}^{\top})-C_{5}W_{1}^{1}+C_{5}W_{2}^{2}\right)\nonumber \\
			& +\frac{1}{2}NW_{n}^{n}A-C(\sum_{i}{F}_{i}^{i}).
		\end{align}
		
		Notice that 
		\begin{align*}
			& -\sigma_{2}({W}^{\top})-C_{5}W_{1}^{1}+C_{5}W_{2}^{2}\\
			= & -W_{1}^{1}W_{2}^{2}-C_{5}W_{1}^{1}+C_{5}W_{2}^{2}\\
			= & W_{1}^{1}(-\frac{1}{2}W_{2}^{2}-C_{5})-W_{2}^{2}(\frac{1}{2}W_{1}^{1}-C_{5}).
		\end{align*}
		
		\emph{Case 2.2.1:} $-\frac{1}{2}W_{2}^{2}-C_{5}\geq0$.
		
		Notice that $\frac{1}{2}W_{1}^{1}-C_{5}\geq-\frac{1}{2}W_{2}^{2}-C_{5}>0$.
		
		Hence, by (\ref{eq:final subcase B.2-1})(\ref{eq:lower bound of big term-1})(\ref{upper bound of F^ii}) we obtain a contradiction via 
		\begin{align*}
			0\ge & F_{j}^{i}G_{i}^{j}\\
			\ge & \frac{13}{25}A^{-\frac{3}{5}}W_{nn}f^{2}+\frac{1}{2}NW_{n}^{n}A-C(\sum_{i}{F}_{i}^{i})\\
			\ge & (\frac{1}{2})^{\frac{13}{8}}W_{n}^{n}N^{\frac{3}{8}}-C(\sum_{i}{F}_{i}^{i})>0,
		\end{align*}
		where we have used the fact that $-\frac{1}{2}W_{2}^{2}-C_{5}\geq0$
		in the second inequality.
		
		\emph{Case 2.2.2:} $-\frac{1}{2}W_{2}^{2}-C_{5}<0$.
		
		Then 
		\begin{equation}
			|W_{2}^{2}|\le W_{1}^{1}\le C^{*}-W_{2}^{2}\le C_{6}.\label{boundness of tangetial derivative in final case}
		\end{equation}
		Notice that $\sigma_{1}(W^{\top})=\frac{f^{2}-\sigma_{2}(W^{\top})+\sum_{\alpha}W_{n}^{\alpha}W_{\alpha}^{n}}{W_{n}^{n}}\le\frac{C}{W_{n}^{n}}.$
		Choosing $N$ large enough such that 
		\[
		f^{2}-\sigma_{2}({W}^{\top})-C_{4}\sigma_{1}({W}^{\top})-\frac{C_{4}W_{1}^{1}}{W_{n}^{n}}\geq\frac{13}{14}\inf_M f^{2}
		\]
		Together with (\ref{eq:final 1-1})\eqref{eq:lower bound of big term-1}\eqref{boundness of tangetial derivative in final case},
		we have 
		\begin{align*}
			0 & \ge F_{j}^{i}G_{i}^{j}\\
			& \ge(\frac{1}{2})^{\frac{13}{8}}W_{n}^{n}N^{\frac{3}{8}}-CN>0.
		\end{align*}
		This is a contradiction.
		
		Therefore, we conclude that the boundary point $x_{0}$ is the maximum
		point of $G$ over $\overline{M_{r}}$. This together with boundary
		condition that $(u_{n}-\varphi)(x_{0})=0$ yields 
		\begin{align*}
			0 & \ge G_{n}(x_{0})=(u_{nn}-\varphi_{n})(x_{0})[1+\frac{7}{5}(u_{n}-\varphi)^{2/5}(x_{0})]-\frac{1}{2}N\\
			& =(u_{nn}-\varphi_{n})(x_{0})-\frac{1}{2}N.
		\end{align*}
		This gives $u_{nn}(x_{0})\le\frac{1}{2}N+\|\varphi\|_{C^{1}(\overline{M_{r}})}$. 
	\end{proof}
	\subsection{Existence of $\sigma_2$-Yamabe equation}
	Without loss of generality,  we assume that the metric $g$ satisfies $h_g=0$ and $\lambda(g^{-1}A_g)\in \Gamma_2^+$ by Theorem \ref{theorem:sigma_2_1st_bdry_eigenvalue-1}. 	
	To prove Therem \ref{theorem:sigma_2_1st_bdry_eigenvalue-1}, we consider the following path of equation:
	\begin{equation}
		\begin{cases}
			\sigma_{2}^{1/2}(g^{-1}(A_{g_{u}}+S_{g}(t)))\\[1pt]
			=(1-t)(\int_{M}e^{-4u})^{\frac{1}{2}}+\psi(t)\sigma_2^{\frac{1}{2}}(g_{\mathbb{S}^{3}}^{-1}A_{\mathbb{S}^{3}})e^{-2u} & \text{in}\quad M^3,\\[3pt]
			\frac{\partial u}{\partial \vec{n}}=0 &\text{on}\quad \partial M,
		\end{cases}\label{eq:path}
	\end{equation}
	where  $S_{g}=(1-\psi(t))(\frac{1}{\sqrt{3}}V_{g}^{1/2}(M)g-A_{g})$
	and $\psi$ is a non-negative function such that $\psi(0)=0$ and $\psi(t)=1$
	for $\frac{1}{2}\le t\le1.$
	\vskip 4pt

	\begin{lemma} \label{lem:finite supervolume with finite interval}Let
		$u$ be a solution to (\ref{eq:path}) for $t \in [0,1]$. Then 
		\[
		(1-t)(\int_{M}e^{-4u}\ud\mu_{g})^{\frac{1}{2}}\le C.
		\]
	\end{lemma} 
	\begin{proof}
		Let $u(x_{0})=\max_{\overline{M}}u(x)$. If $x_0\in M$, then  we know $\nabla u(x_{0})=0$ and $\nabla^{2}u(x_{0})\le0$.
		Thus, we have 
		\[
		(1-t)(\int_{M}e^{-4u}\ud\mu_{g})^{\frac{1}{2}}\le\sigma_{2}^{1/2}(A_{g}(x_{0})+S_{g}(x_{0})).
		\]
		The desired estimate follows. 
		If $x_0\in \partial M$, then 
		$0\ge \frac{\partial u}{\partial\vec{n}}(x_0)=0$. So  $\nabla u(x_{0})=0$ and $\nabla^{2}u(x_{0})\le0$ on $\pa M$. The remaining argument is similar as above. 
	\end{proof}
	
	\begin{thm}
		\label{C0estimate} Assume 
		\begin{equation}
			\Lambda_{2}(M^{3},\partial M, [g])<\frac{1}{2}\mathrm{vol}(\mathbb{S}^{3},g_{\mathbb{S}^{3}}).\label{upper bound of conformal invariant}
		\end{equation}
		Let $u$ be a  $C^4$ solution to (\ref{eq:path}). Then there exists a uniform
		constant $C$ such that 
		\[
		|u|_{C^{0}(M)}\le C.
		\]
	\end{thm}
	
	The proof of Theorem \ref{C0estimate} is nearly identical to that of \cite[Proposition 2]{Chen}. We skip most of the argument and emphasize that the only significant difference lies in the contradiction derived from (\ref{upper bound of conformal invariant}) instead of the Gauss-Bonnet-Chern formula for manifolds with umbilic boundary.
	\begin{proof}
		\emph{Step A}: we prove that \emph{(A.1)}: $\inf_{M}u\ge-C$ for $t\in[0,1-\varepsilon]$; 
		\emph{(A.2)}: $\inf_{M}u>-C$ for $t\rightarrow1.$ The proof of \emph{(A.1)}
		is exactly same as the argument of S. Chen \cite{Chen} and we omit it. For \emph{(A.2)},
		we modify the argument of \cite{Chen} and suppose that there exists a sequence
		of solutions $u_{i}$ to (\ref{eq:path}) and $t_{i}\rightarrow1$
		such that $\inf_{M}u_{i}=u_{i}(p_{i})\rightarrow-\infty$, $p_{i}\rightarrow p_{0}$, and $\psi(t_{i})=1$ as $t_{i}\rightarrow1$. Let
		$\varepsilon_{i}=e^{\inf u_{i}}\rightarrow0$ and $d_{i}$ be the
		distance from $p_{i}$ to $\partial M$. 
		
		Denote 
		\[
		\tilde{u}_{i}(x)=u_{i}(\exp_{p_{i}}\varepsilon_{i}x)-\log\varepsilon_{i}
		\]
		and 
		\[
		g_{i}=\varepsilon_{i}^{-2}\psi_{i}^{*}(g)=g_{i,kl}(\psi_{i}(x))\ud x^{k}\otimes \ud x^{l}\rightarrow g_{\mathrm{E}},\quad\text{as}\quad i\rightarrow\infty,
		\]
		where $\psi_{i}(x)=\exp_{p_{i}}(\varepsilon_{i}x)$ and $g_{\mathrm E}$ is the Euclidean flat metric. We divide into
		two cases: $\frac{d_{i}}{\varepsilon_{i}}\rightarrow\infty$ and $\frac{d_{i}}{\varepsilon_{i}}\le C_{0}.$
		By Theorem \ref{Thm:reduction} and Theorem \ref{three-double normal}, we know that $|\tilde{u}_{i}|_{C^2(K)}\le C$  for any compact set $K\subset \overline{\mathbb{R}_{+}^{3}}$ if $\frac{d_{i}}{\varepsilon_{i}}\le C_{0}$ or  $K\subset {\mathbb{R}^{3}}$ if  $\frac{d_{i}}{\varepsilon_{i}}\rightarrow\infty$. Then
		the rescaled function $\tilde{u}_{i}$ converges uniformly to a solution
		$u\in C^{\infty}(\overline{\mathbb{R}_{+}^{3}})$ to 
		\begin{equation}
			\begin{cases}
				\sigma_{2}^{1/2}(\nabla^{2}u+\nabla u\otimes\nabla u-\frac{|\nabla u|^{2}}{2}g_{ \mathrm E})=\sigma_2^{\frac{1}{2}}(g_{\mathbb{S}^{3}}^{-1}A_{\mathbb{S}^{3}})e^{-2u} &\text{in}\quad \overline{\mathbb{R}_{+}^{3}},\\
				\frac{\partial u}{\partial \vec{n}}=0 &\text{on}\quad \partial\overline{\mathbb{R}_{+}^{3}},
			\end{cases}\label{eq:A.2.1}
		\end{equation}
		or 
		\begin{equation}
			\sigma_{2}^{1/2}(\nabla^{2}u+\nabla u\otimes\nabla u-\frac{|\nabla u|^{2}}{2}g_{\mathrm E})=\sigma_2^{\frac{1}{2}}(g_{\mathbb{S}^{3}}^{-1}A_{\mathbb{S}^{3}})e^{-2u}\quad\text{in}\quad\mathbb{R}^{3}.\label{eq:A.2.2}
		\end{equation}
By the Liouville theorem in \cite{Li-Li2, Li-Li3}, the solution $u$ has been classified as sphere or half sphere.
		For (\ref{eq:A.2.1}), fix any $r_{0}>0$, 
		\begin{align*}
			\mathrm{vol}(e^{-2u_{i}}g) & =\int_{M}e^{-3u_{i}}\ud V_{g}\ge\int_{B_{r_{0}}^{+}(x_{i})}e^{-3u_{i}}\ud V_{g}\\
			& =\int_{\psi_{i}^{-1}(B_{r_{0}}^{+}(x_{i}))}e^{-3\widetilde{u}_{i}}\ud V_{g_{i}}.
		\end{align*}
		Thus,  
		we have
		\begin{equation}\label{halfvolume}
		\underline{\lim}_{i\rightarrow\infty}\mathrm{vol}(e^{-2u_{i}}g)\ge\frac{1}{2}\mathrm{vol}(\mathbb{S}^{3}).
		\end{equation}
		
		By (\ref{eq:A.2.2}), 
		\begin{align*}
			\mathrm{vol}(e^{-2u_{i}}g) & =\int_M e^{-3u_{i}}\ud V_{g}\ge\int_{B_{r_{0}}(x_{i})}e^{-3u_{i}}\ud V_{g}\\
			& =\int_{\psi_{i}^{-1}(B_{r_{0}}(x_{i}))}e^{-3\widetilde{u}_{i}}\ud V_{g_{i}},
		\end{align*}
		thus 
		\begin{equation}\label{volume}
		\underline{\lim}_{i\rightarrow\infty}\mathrm{vol}(e^{-2u_{i}}g)\ge \mathrm{vol}(\mathbb{S}^{3}).
		\end{equation}
		As  $\sigma_{2}^{1/2}(g_{u_{i}}^{-1}A_{g_{u_{i}}})\ge\sigma_{2}^{1/2}(g_{\mathbb{S}^{3}}^{-1}A_{\mathbb{S}^{3}})$ for large $i$,
		\eqref{halfvolume} or  \eqref{volume} contradicts the assumption (\ref{upper bound of conformal invariant})
		
		\emph{Step B}: We prove that $\sup_{M}u\le C.$ The proof is identical to that of S. Chen \cite{Chen} and  $\sup_{M}u$
		can be finished by the local gradient estimate  by Jin-Li-Li \cite{Jin-Li-Li}. 
	\end{proof}
	
	The existence part can be deduced from the standard
	argument in S. Chen \cite{Chen}, Gursky-Viaclovsky \cite{GV0}.  Consequently, we prove Theorem \ref{thm:three-double normal}.

\end{document}